\documentclass[12pt]{amsart}
\usepackage{amscd, amsfonts, amssymb}
\usepackage{epsfig}
\usepackage[mathscr]{eucal}
\usepackage{verbatim}
\usepackage{fullpage}
\usepackage{latexsym}
\usepackage{lscape}
\usepackage{enumerate}

\newcommand{\mc}{\mathcal}

    \renewcommand\gg{\mathfrak g}
\newcommand\hh{\mathfrak h}

\newcommand\frakl{\mathfrak l}
\newcommand\mm{\mathfrak m}

\newcommand\pp{\mathfrak p}

\newcommand\inverse{{^{-1}}}

\newcommand\ra{\rightarrow}

\DeclareMathOperator{\Ad}{Ad}
\DeclareMathOperator{\Aut}{Aut}

\DeclareMathOperator{\GL}{GL}
\DeclareMathOperator{\Mat}{Mat}
\DeclareMathOperator{\Gal}{Gal}
\DeclareMathOperator{\SL}{SL}

\DeclareMathOperator{\Lie}{Lie}

\DeclareMathOperator{\End}{End}

\newcommand{\twobytwo}[4]
 {{\left(\begin{array}{ll} #1 & #2 \\ #3 & #4 \end{array}\right)}}
\newcommand{\tuple}[1]{{\mathbf {#1}}}

\numberwithin{equation}{section}

\newtheorem{thm}[equation]{Theorem}

\newtheorem{lem}[equation]{Lemma}
\newtheorem{cor}[equation]{Corollary}
\newtheorem{prop}[equation]{Proposition}
\newtheorem{conj}[equation]{Conjecture}

\theoremstyle{definition}
\newtheorem{defn}[equation]{Definition}
\newtheorem{exmp}[equation]{Example}

\theoremstyle{remark}
\newtheorem{rem}[equation]{Remark}

\theoremstyle{remark}
\newtheorem{rems}[equation]{Remarks}
\newtheorem{question}[equation]{Question}

\newcommand{\ovl}[1]{{\overline{#1}}}


\title[Closed Orbits and uniform $S$-instability]
{Closed Orbits and uniform $S$-instability \\ in Geometric Invariant Theory}

\author[M.\  Bate]{Michael Bate}
\address
{Department of Mathematics,
University of York,
York YO10 5DD,
United Kingdom}
\email{michael.bate@york.ac.uk}

\author[B.\ Martin]{Benjamin Martin}
\address
{Mathematics and Statistics Department,
University of Canterbury,
Private Bag 4800,
Christchurch 1,
New Zealand}
\email{B.Martin@math.canterbury.ac.nz}

\author[G. R\"ohrle]{Gerhard R\"ohrle}
\address
{Fakult\"at f\"ur Mathematik,
Ruhr-Universit\"at Bochum,
D-44780 Bochum, Germany}
\email{gerhard.roehrle@rub.de}

\author[R.\ Tange]{Rudolf Tange}
\address
{School of Mathematics,
Trinity College Dublin,
College Green,
Dublin 2,
Ireland}
\email{tanger@tcd.ie}

\begin{document}

\begin{abstract}
In this paper we consider various problems involving
the action of a reductive group $G$ on an affine variety $V$.
We prove some general rationality
results about the $G$-orbits in $V$.
In addition, we extend fundamental results of Kempf and Hesselink regarding
optimal destabilizing parabolic subgroups of $G$ for such
general $G$-actions.

We apply our general rationality results
to answer a question of Serre
concerning the behaviour of his notion of $G$-complete reducibility
under separable field extensions.
Applications of our new optimality results also
include a construction which allows
us to associate an optimal destabilizing parabolic subgroup of $G$
to any subgroup of $G$.
Finally, we use these new optimality techniques to provide an
answer to Tits' Centre Conjecture in a special case.
\end{abstract}

\maketitle

\tableofcontents

\section{Introduction}
\label{sec:intro}

Let $G$ be a reductive linear algebraic group over an algebraically closed field $k$ and suppose $G$ acts on an affine variety $V$ over $k$.  A fundamental problem in geometric invariant theory is to understand the structure of the $G$-orbits $G\cdot v$ and their closures $\ovl{G\cdot v}$ for $v\in V$.  It is well known that $\ovl{G\cdot v}$ is a union of $G$-orbits, exactly one of which is closed.  Moreover, the Hilbert-Mumford Theorem \cite[Thm.~1.4]{kempf} tells us that if $G\cdot v$ is not closed, then there exists a cocharacter $\lambda$ of $G$ such that
$\lim_{a\ra 0} \lambda(a)\cdot v$ exists and lies outside ${G\cdot v}$.  One can associate to $\lambda$ a parabolic subgroup $P_\lambda$ of $G$; we call $\lambda$ a \emph{destabilizing cocharacter} for $v$ and we call $P_\lambda$ a \emph{destabilizing parabolic subgroup} for $v$.

A strengthened version of the Hilbert-Mumford Theorem --- due to Kempf \cite{kempf} and Rousseau \cite{rousseau} --- says that there exists a so-called \emph{optimal cocharacter} $\lambda_v$ such that $\lim_{a\ra 0} \lambda_v(a)\cdot v$ exists and lies outside $G\cdot v$, and such that $\lambda_v$ takes $v$ outside $G\cdot v$ ``as fast as possible''.  This optimality notion has had several applications, including to $G$-complete reducibility \cite{BMR}, \cite{BMR2}, \cite{BMRT} and the theory of associated cocharacters for nilpotent elements of $\Lie G$ \cite{jantzen}, \cite{premet}; see \cite[p64 and App.~2B]{mumford} for further discussion. Hesselink used optimality to study the nullcone of a rational $G$-module \cite{He2}.

In this paper we investigate the structure of the orbits when the field $k$ is not algebraically closed.  Little seems to be known here.  Indeed, one of Kempf's motivations for his optimality construction was to prove a rationality result for destabilizing cocharacters over a perfect field \cite[Thm.~4.2]{kempf}.  In Sections \ref{sec:orbrat} and \ref{sec:uniform} we prove some results in the general setting of geometric invariant theory.  In Section \ref{sec:appl-gcr} we give applications to the theory of $G$-complete reducibility and Tits' Centre Conjecture.  Below we describe the contents of the paper in more detail.

It is convenient to extend the concept of orbit closure to the non-algebraically closed case.  We say that the
$G(k)$-orbit $G(k)\cdot v$ is \emph{cocharacter-closed over $k$}
if for any $k$-defined cocharacter $\lambda $ of $G$
such that $v':= \lim_{a\ra 0} \lambda(a)\cdot v$ exists,
$v'$ is $G(k)$-conjugate to $v$ (see Definition \ref{def:cocharclosure}).
Clearly, this notion depends only on the $G(k)$-orbit
$G(k)\cdot v$ of $v$ and not on $v$ itself.  Let $\ovl{k}$ denote the algebraic closure of $k$.  It follows from the Hilbert-Mumford Theorem that
$G\cdot v$ is closed if and only if $G(\ovl{k})\cdot v$
is cocharacter-closed over $\ovl{k}$.
It is sensible, therefore, to consider the $G(k)$-orbits
that are cocharacter-closed over $k$ as a  generalization
to non-algebraically closed fields $k$ of the closed $G$-orbits.
Some of the ideas in Section 3 of this paper were studied by J.~Levy 
in the special case of characteristic 0, cf.~\cite{levy}; 
we thank Levy for drawing our attention to \cite{levy}. 

Understanding the structure of the orbits is a delicate problem because the interplay between the $G$-orbits and the $G(k)$-orbits is quite complicated.  Let $v\in V(k)$.  Suppose first that $G\cdot v$ is not closed and let $S$ be the unique closed $G$-orbit contained in $\ovl{G\cdot v}$.  It can happen that $G(k)\cdot v$ is cocharacter-closed over $k$, so there need not exist a $k$-defined cocharacter $\lambda$ such that $\lim_{a\ra 0} \lambda(a)\cdot v$ exists and belongs to $S$; indeed, $S$ need not have any $k$-points at all.  Now suppose that $G\cdot v$ is closed.  If $\lambda$ is a $k$-defined cocharacter, then it can happen that $\lim_{a\ra 0} \lambda(a)\cdot v$ exists but lies outside $G(k)\cdot v$: in this case, $G(k)\cdot v$ is not cocharacter-closed over $k$.  We give concrete examples of these phenomena in Remark \ref{rem:obstacle} (see also Question \ref{qn:cocharclsdext}).

Our work on geometric invariant theory has two main strands.  Let $v\in V(k)$ and let $\lambda$ be a $k$-defined cocharacter of $G$ such that $v':= \lim_{a\ra 0} \lambda(a)\cdot v$ exists.  First we consider the case when $v'$ lies in $G(k)\cdot v$.  Our main results here are Theorems \ref{thm:Ruconj} and \ref{thm:cocharclosedcrit}, which show that under some additional hypotheses, $v'$ lies in $R_u(P_\lambda)(k)\cdot v$.  Theorem \ref{thm:Ruconj} was first proved by H.~Kraft and J.~Kuttler for $k$ algebraically closed of characteristic zero in case $V = G/H$ is an affine homogeneous space, cf.\
\cite[Prop.\ 2.1.4]{schmitt} or \cite[Prop.\ 2.1.2]{Gomez}.

Second, we consider the case when $v'$ lies outside $G\cdot v$.  We extend work of Kempf and Hesselink on optimality.
In \cite{kempf}, Kempf shows that if
$v \in V$ is a point whose $G$-orbit $G\cdot v$ is not closed,
and $S$ is a $G$-stable closed subvariety of $V$ which meets the
closure of $G\cdot v$, then there is an optimal class of cocharacters
which move $v$ into $S$ (by taking limits).
In a similar vein, in \cite{He} Hesselink develops a notion of \emph{uniform
instability}: here the single point $v \in V$ in
Kempf's construction is replaced by a subset $X$ of $V$, but the
$G$-stable subvariety $S$ is taken to be a single point of $V$.  Moreover, Hesselink's results work for arbitrary non-algebraically closed fields.
Our constructions, culminating in Theorem \ref{thm:kempfrousseau}, combine
these two ideas within the single framework of
\emph{uniform $S$-instability}, providing a useful extension
of these optimality methods in geometric invariant theory.

There is an important open problem which we do {\bf not} address.  We do not deal with the intermediate case when $v'$ lies inside $G\cdot v$ but outside $G(k)\cdot v$: in particular, our optimality results do not give a true generalization of the Hilbert-Mumford-Kempf-Rousseau optimality theorem to arbitrary $k$.  To do this, one would have to answer the following question.  Suppose $v\in V(k)$ and there exists a $k$-defined cocharacter $\lambda$ such that $\lim_{a\ra 0} \lambda(a)\cdot v$ exists and lies outside $G(k)\cdot v$.  Does there exist an optimal $k$-defined cocharacter which takes $v$ outside $G(k)\cdot v$ as fast as possible?  The cocharacter $\lambda_v$ described above will not suffice: for instance, if $G\cdot v$ is closed, then $\lambda_v$ is not even defined.  We plan to return to this question in future work.

The hypothesis that the point $v\in V$ is a $k$-point turns out to be unnecessarily strong, and we can often get away with a weaker condition on the stabilizer $C_G(v)$ (see the beginning of Section \ref{sec:orbrat}).  This is convenient in applications to $G$-complete reducibility (see Remark~\ref{rem:generictuple}).

As well as being of interest in their own right, our general results on $G$-orbits and
rationality have applications to the theory of $G$-complete
reducibility, introduced by Serre \cite{serre1.5} and developed in
\cite{BMR}, \cite{BMR2}, \cite{BMRT}, \cite{BMRT:relative}, \cite{liebeckseitz0}, \cite{liebeckseitz}, \cite{liebecktesterman}, \cite{seitz1}, \cite{serre0}, \cite{serre1}, \cite{serre2}.
In particular, we are able to use them to answer a question of Serre about how $G$-complete
reducibility behaves under extensions of fields
(Theorem \ref{thm:serresquestion}).
Our notion of a cocharacter-closed orbit allows us to give a geometric characterization of $G$-complete reducibility over a field $k$ (Theorem~\ref{thm:cocharclosedcritforGcr}), thereby extending \cite[Cor.~3.7]{BMR}.
We use our optimality results to attach to any subgroup $H$ of
$G$ an optimal parabolic subgroup of $G$ containing $H$,
which is proper if and only if $H$ is not $G$-completely reducible
(see Theorem~\ref{thm:optpar} and Definition \ref{defn:optpar}).
This optimal parabolic subgroup provides a very
useful tool in the study of subgroups of reductive groups.
As an illustration of its effectiveness, we give short proofs of some existing
results, and prove a special case of Tits' Centre Conjecture
(Theorem~\ref{thm:X^N}).  An important tool, which we introduce in Definition~\ref{def:generictuple}, is the notion of a generic tuple of a subgroup $H$ of $G$.  Replacing generating tuples with generic tuples allows us to avoid many technical problems that arose in our earlier work (see Remark~\ref{rem:generictuple}).

We also refer the reader to \cite{BMRT:relative}, where we discuss further consequences of the results of the present paper.

\section{Notation and preliminaries}
\label{sec:prelims}

\subsection{Basic notation}

Let $k$ be a field, let $k_s$ denote its separable closure, and let $\ovl k$
denote its algebraic closure. Note that $k_s=\ovl{k}$ if $k$ is perfect.
We denote the Galois group $\Gal(k_s/k)=\Gal(\ovl k/k)$ by $\Gamma$.
We use the notion of a $k$-scheme from \cite[AG.11]{Bo}: a $k$-scheme is a
$\ovl k$-scheme together with a $k$-structure.
So $k$-schemes are assumed to be of finite type and reduced
separated $k$-schemes are called $k$-varieties.
Furthermore, a subscheme of a scheme $V$ over $k$ or over
$\ovl{k}$ is always a subscheme of $V$
as a scheme over $\ovl{k}$ and points of $V$ are always
closed points of $V$ as a
scheme over $\ovl{k}$.  By ``variety'' we mean ``variety over $\ovl{k}$''.
Non-reduced schemes are only used in
Section~\ref{sec:uniform} and there they
only play a technical r\^ole; we always formulate our results
for $k$-varieties.  If $S$ is a subset of a variety,
then $\ovl{S}$ denotes the closure of $S$.

Now let $V$ be a $k$-variety.  If $k_1/k$ is an algebraic extension,
then we write $V(k_1)$ for the set of $k_1$-points of $V$.
By a \emph{separable point} we mean a $k_s$-point.
If $W$ is a subvariety of $V$, then we set $W(k_1)= W\cap V(k_1)$.
Here we do not assume that $W$ is $k$-defined, so $W(k_1)$
can be empty even when $k_1= k_s$.  The Galois group $\Gamma$
acts on $V$; see, e.g., \cite[11.2]{spr2}.
Recall the Galois criterion for a closed subvariety $W$ of $V$
to be $k$-defined: $W$ is $k$-defined if and only if
it contains a
$\Gamma$-stable set of separable points of $V$ which is dense in $W$
(see \cite[Thm.~AG.14.4]{Bo}).

We denote by ${\rm Mat}_m$ or ${\rm Mat_m}(k)$ the algebra of $m\times m$ matrices over $k$.  The general linear group ${\rm GL}_m$ acts on ${\rm Mat}_m$ by conjugation.

Let $H$ be a $k$-defined linear algebraic group.  By a subgroup of $H$ we mean a closed subgroup.
We let $Z(H)$ denote the centre of $H$ and $H^0$ the connected component of
$H$ that contains $1$.  Recall that $H$ has a $k$-defined maximal torus \cite[18.2(i)~Thm.]{Bo}.
For $K$ a subgroup of $H$, we denote the
centralizer of $K$ in $H$ by $C_H(K)$ and the normalizer of $K$ in $H$
by $N_H(K)$.
We denote the group of algebraic automorphisms of $H$ by $\Aut H$.

For the set of cocharacters (one-parameter subgroups) of $H$ we write $Y(H)$;
the elements of $Y(H)$ are the homomorphisms from the multiplicative group
$\ovl{k}^*$ to $H$. We denote the set of $k$-defined
cocharacters by $Y_k(H)$.
There is a left action of $H$ on $Y(H)$ given by
$(h\cdot \lambda)(a) = h\lambda(a)h^{-1}$ for
$\lambda\in Y(H)$, $h\in H$ and $a \in \ovl{k}^*$.  The subset $Y_k(H)$ is stabilized by $H(k)$.

The \emph{unipotent radical} of $H$ is denoted $R_u(H)$; it is the maximal
connected normal unipotent subgroup of $H$.
The algebraic group $H$ is called \emph{reductive} if $R_u(H) = \{1\}$;
note that we do not insist that a reductive group is connected.

Let $A$ be an algebraic group, a Lie algebra or an associative algebra.  If $n\in {\mathbb N}$ and $\tuple{x}= (x_1,\ldots, x_n)\in A^n$, then we say that $\tuple{x}$ generates $A$ if the $x_i$ generate $A$ as an algebraic group (resp.\ Lie algebra, resp.\ associative algebra).  By this we mean in the algebraic group case that the algebraic subgroup of $A$ generated by the $x_i$ is the whole of $A$, and we say that the algebraic group $A$ is topologically finitely generated.

Throughout the paper, $G$ denotes a $k$-defined reductive algebraic group, possibly
disconnected.  We say an affine $G$-variety $V$ is $k$-defined if both $V$ and the action of $G$ on $V$ are $k$-defined.
By a rational $G$-module, we mean a finite-dimensional vector space over $\ovl{k}$ with a linear $G$-action.  If both $V$ and the action are $k$-defined, then we say the rational $G$-module is $k$-defined.

Suppose $T$ is a maximal torus of $G$.
Let $\Psi = \Psi(G,T)$ be the set of roots of $G$ relative to $T$.
Let $\alpha\in \Psi$. Then $U_\alpha$ denotes
the root subgroup of $G$ associated to $\alpha$.

\subsection{Non-connected reductive groups}
\label{subsec:noncon}

The crucial idea which allows us to deal with non-connected groups is the introduction of so-called
\emph{Richardson parabolic subgroups} (\emph{R-parabolic subgroups}) of a
reductive group $G$.
We briefly recall the main definitions and results;
for more details and further results,
the reader is referred to \cite[Sec.\ 6]{BMR}.

\begin{defn}
\label{defn:rpars}
For each cocharacter $\lambda \in Y(G)$, let
$P_\lambda = \{ g\in G \mid \underset{a \to 0}{\lim}\,
\lambda(a) g \lambda(a)\inverse \textrm{ exists} \}$ (see Section~\ref{subsec:Gvars} for the definition of limit).
Recall that a subgroup $P$ of $G$ is \emph{parabolic}
if $G/P$ is a complete variety.
The subgroup $P_\lambda$ is parabolic in this sense, but the converse
is not true: e.g.,\ if $G$ is finite,
then every subgroup is parabolic, but the only
subgroup of $G$ of the form $P_\lambda$ is $G$ itself.
If we define
$L_\lambda = \{g \in G \mid \underset{a \to 0}{\lim}\,
\lambda(a) g \lambda(a)\inverse = g\}$,
then $P_\lambda = L_\lambda \ltimes R_u(P_\lambda)$,
and we also have
$R_u(P_\lambda) = \{g \in G \mid \underset{a \to 0}{\lim}\,
\lambda(a) g \lambda(a)\inverse = 1\}$.
The map $c_\lambda :P_\lambda \to L_\lambda$ given by
$c_\lambda(g) = \underset{a\to 0}{\lim}\, \lambda(a)g\lambda(a)\inverse$ is
a surjective homomorphism of algebraic groups with kernel $R_u(P_\lambda)$;
it coincides with the usual projection $P_\lambda\ra L_\lambda$.  We abuse notation and denote the corresponding map from $P_\lambda^n$ to $L_\lambda^n$ by $c_\lambda$ as well, for any $n\in {\mathbb N}$.
The subgroups $P_\lambda$ for $\lambda \in Y(G)$
are called the \emph{Richardson parabolic} (or \emph{R-parabolic}) \emph{subgroups}
of $G$.
Given an R-parabolic subgroup $P$,
a \emph{Richardson Levi} (or \emph{R-Levi})
\emph{subgroup} of $P$ is any subgroup $L_\lambda$
such that $\lambda \in Y(G)$ and $P=P_\lambda$.
\end{defn}

If $G$ is connected, then the R-parabolic subgroups
(resp.\ R-Levi subgroups of R-parabolic subgroups)
of $G$ are exactly the parabolic subgroups
(resp.\ Levi subgroups of parabolic subgroups) of $G$;
indeed, most of the theory of parabolic subgroups and
Levi subgroups of connected reductive groups
carries over to R-parabolic and R-Levi subgroups of
arbitrary reductive groups.
In particular, all R-Levi subgroups of an R-parabolic
subgroup $P$ are conjugate under the action
of $R_u(P)$.  If $P,Q$ are R-parabolic subgroups of $G$ and $P^0= Q^0$,
then $R_u(P)= R_u(Q)$.

\begin{lem}
\label{lem:Levidown}
Let $P,Q$ be R-parabolic subgroups of $G$
with $P\subseteq Q$ and $P^0= Q^0$, and
let $M$ be an R-Levi subgroup of $Q$.
Then $P\cap M$ is an R-Levi subgroup of $P$.
\end{lem}

\begin{proof}
Fix a maximal torus $T$ of $G$ such that $T\subseteq M$.
Then $T \subseteq P$, since $P^0= Q^0$.
There exists a unique R-Levi subgroup $L$ of $P$ such that
$T\subseteq L$, \cite[Cor.~6.5]{BMR}.  There exists a unique
R-Levi subgroup $M'$ of $Q$ such that $L\subseteq M'$, \cite[Cor.~6.6]{BMR}.
Since $M$ is the unique R-Levi subgroup of $Q$ that contains $T$,
\cite[Cor.~6.5]{BMR}, we must have $M= M'$.  Hence $L\subseteq P\cap M$.
If this inclusion is proper, then $P\cap M$ meets $R_u(P)= R_u(Q)$
non-trivially, a contradiction.  We deduce that $L= P\cap M$.
\end{proof}

We now consider some rationality issues.  The proof of the next
lemma follows immediately from the definitions of limit and of
the actions of $\Gamma$ on $k_s$-points and on $k_s$-defined morphisms.

\begin{lem}
\label{lem:Galoislim}
 Let $\lambda\in Y_{k_s}(G)$ and let $\gamma\in \Gamma$.
Then $P_{\gamma\cdot \lambda}= \gamma\cdot P_\lambda$ and $L_{\gamma\cdot \lambda}= \gamma\cdot L_\lambda$.
\end{lem}

\begin{rem}
\label{rem:krpars}
If $G$ is connected, then
a parabolic subgroup $P$ of $G$ is $k$-defined
if and only if $P=P_\lambda$ for some $\lambda \in Y_k(G)$,
\cite[Lem.~15.1.2(ii)]{spr2}.
However, the analogous result for R-parabolic subgroups of
a non-connected group $G$ is not true in general.
To see this, let $T$ be a
non-split one-dimensional torus over $k$ and let $F$ be the group of order $2$
acting on $T$ by inversion. Then $T$ is a $k$-defined R-parabolic subgroup of
the reductive group $G:=FT$,
but $T$ is not of the form $P_\lambda$ for any $\lambda$ over $k$, because $Y_k(G) = \{0\}$.
Our next set of results allow us to deal with this problem.
\end{rem}

\begin{lem}
\label{lem:Rparrat}
Let $\lambda\in Y(G)$.
\begin{enumerate}[{\rm(i)}]
\item  If $P_\lambda$ is $k$-defined, then so is $R_u(P_\lambda)$.
Moreover, if $\lambda$ belongs to $Y_k(G)$, then $P_\lambda$, $L_\lambda$
and the isomorphism $L_\lambda\ltimes R_u(P_\lambda)\ra P_\lambda$
are $k$-defined.
\item Suppose $P_\lambda$ is $k$-defined.
Then there
exists $\mu\in Y_k(G)$ such that $P_\lambda\subseteq P_\mu$ and
$P_\lambda^0=P_\mu^0$.
\item Let $P$ be a $k$-defined $R$-parabolic subgroup.
Then any $k$-defined
maximal torus of $P$ is contained in a unique $k$-defined
R-Levi subgroup of $P$ and any two
$k$-defined R-Levi subgroups of $P$ are conjugate by a
unique element of $R_u(P)(k)$.
\end{enumerate}
\end{lem}

\begin{proof}
(i).\ We have that $P_\lambda^0$ is $k$-defined,
so $R_u(P_\lambda) = R_u(P_\lambda^0)$ is
$k$-defined, by \cite[Prop.~V.20.5]{Bo}.
Now assume that $\lambda\in Y_k(G)$.
Then $L_\lambda= C_G(\lambda(\ovl{k}^*))$
is defined over $k$, by \cite[Cor.~III.9.2]{Bo}.
Now the multiplication map $G\times G\ra G$ is $k$-defined,
so $P_\lambda$ is $k$-defined, thanks to \cite[Cor.\ AG.14.5]{Bo},
and the stated isomorphism is then clearly also $k$-defined.

(ii).\ After conjugating $\lambda$ by an element of
$P_\lambda$, we may assume that
$\lambda\in Y(T)$ for some $k$-defined maximal torus $T$ of $P_\lambda$.
Since $T$ splits over a finite Galois extension of $k$, $\lambda$ has only finitely many
$\Gamma$-conjugates. Let $\mu\in Y(T)$ be their sum. Since $P_\lambda$ is $k$-defined, we have
$P_{\gamma\cdot\lambda}=P_\lambda$ for all $\gamma\in\Gamma$. By considering the
pairings of $\lambda$ and $\mu$ with the coroots of $G$ relative to $T$, we deduce that
$P_\mu^0=P_\lambda^0$ (cf.\ \cite[15.1.2]{spr2}). Using a $G$-equivariant embedding of $G$ acting on itself
by conjugation into a finite-dimensional $G$-module, we deduce that
$\underset{a\to 0}{\lim}\,  \mu(a)\cdot g$ exists if
$\underset{a\to 0}{\lim}\,  (\gamma\cdot\lambda)(a)\cdot g$ exists for all $\gamma\in\Gamma$.
So $P_\lambda\subseteq P_\mu$.

(iii).\ Because of \cite[Prop.~V.20.5]{Bo}
and \cite[Cors.~6.5, 6.6, 6.7]{BMR}, it is enough
to show that the unique R-Levi subgroup
of $P$ containing a given $k$-defined maximal torus
of $P$ is $k$-defined.
(Here the required uniqueness follows from \emph{loc.\ cit.})
Let $T$ be a $k$-defined maximal torus
of $P$. 
By the proof of (ii),
there exists $\mu\in Y_k(T)$ such that $P\subseteq P_\mu$ and
$P^0=P_\mu^0$.
Clearly, $L_\mu$ is the R-Levi subgroup of $P_\mu$ containing $T$, and
it is $k$-defined by (i). The unique R-Levi subgroup of $P$ containing $T$ is $P\cap L_\mu$, by Lemma \ref{lem:Levidown}.
Since $P\cap G(k_s)$ and $L_\mu\cap G(k_s)$ are $\Gamma$-stable, the same holds for
$P\cap L_\mu\cap G(k_s)$. So it suffices to show that this set is dense in $P\cap L_\mu$.
This follows because the components of $P\cap L_\mu$ are components of $L_\mu$ and the
separable points are dense in each component of $L_\mu$.
\end{proof}

\begin{cor}
\label{cor:kLevi}
Let $\lambda\in Y_k(G)$ and let $\mu\in Y(G)$ such that $P_\lambda=P_\mu$
and $L_\mu$ is $k$-defined.
Then there exists $\nu\in Y_k(G)$ such that $P_\lambda=P_\nu$ and $L_\mu=L_\nu$.
\end{cor}

\begin{proof}
By Lemma~\ref{lem:Rparrat}(iii), there exists
$u\in R_u(P_\lambda)(k)$ such that
$L_{u\cdot\lambda}=uL_\lambda u\inverse=L_\mu$, so we can take $\nu=u\cdot\lambda$.
\end{proof}

\subsection{$G$-varieties}
\label{subsec:Gvars}
If $G$ acts on a set $V$, then we denote for a subset $S$ of $V$, the pointwise stabilizer $\{g\in G \mid g\cdot s= s \textrm{ for all } s\in S\}$ of $S$ in $G$ by $C_G(S)$ and the setwise stabilizer $\{g\in G \mid g\cdot S = S\}$ of $S$ in $V$ by $N_G(S)$.

Now suppose $G$ acts on an affine variety $V$ and let $v \in V$.
Then for each cocharacter $\lambda \in Y(G)$, we can
define a morphism of varieties
$\phi_{v,\lambda}:\ovl{k}^* \to V$ via the formula
$\phi_{v,\lambda}(a) = \lambda(a)\cdot v$.
If this morphism extends to a morphism
$\widehat\phi_{v,\lambda}:\ovl{k} \to V$, then
we say that $\underset{a\to 0}{\lim}\, \lambda(a) \cdot v$ exists,
and set this limit
equal to $\widehat\phi_{v,\lambda}(0)$; note that such an extension,
if it exists, is necessarily unique.

Let $\lambda\in Y(G)$. Then the set of $v\in V$ such that $\underset{a\to 0}{\lim}\, \lambda(a) \cdot v$ exists
is $P_\lambda$-stable and we have
\begin{equation}
\label{eqn:limx.v}
 \lim_{a\ra 0} \lambda(a)\cdot (x\cdot v)= c_\lambda(x)\cdot \left(\lim_{a\ra 0} \lambda(a)\cdot v\right),
\end{equation}
for all $x\in P_\lambda$ and $v\in V$.
Suppose that the $G$-variety $V$ is $k$-defined.
It is easily shown that if $\phi_{v,\lambda}$ is $k$-defined,
then $\widehat\phi_{v,\lambda}$ is $k$-defined and $\lim_{a\ra 0} \lambda(a)\cdot v\in V(k)$; in particular,
this is the case if $\lambda\in Y_k(G)$ and $v\in V(k)$.

\begin{rem}
\label{rem:linear}
In many of our proofs,
we want to reduce the case of a general ($k$-defined) affine $G$-variety $V$
to the case of a ($k$-defined)
rational $G$-module $V_0$.
Such a reduction is possible, thanks to \cite[Lem.~1.1(a)]{kempf},
for example: given $V$, there is a $k$-defined $G$-equivariant embedding of $V$ inside some $V_0$.
As this situation arises many times in the sequel, we now set up some
standard notation which will be in force throughout the paper.

Let $V$ be a rational $G$-module.
Given $\lambda \in Y(G)$ and $n\in\mathbb{Z}$, we define
\begin{align}
\label{eq:defvlambda0}
V_{\lambda,n} :=
\{v\in V\mid \lambda(a)\cdot v=a^nv\text{\ for all\ }a\in \ovl{k}^*\},\\
V_{\lambda, \ge0}:=\sum_{n\ge0} V_{\lambda,n}
\quad
\textrm{ and }
\quad
V_{\lambda, >0}:=\sum_{n>0} V_{\lambda,n}.\notag
\end{align}
Then $V_{\lambda,\ge0}$ consists of the vectors
$v\in V$ such that $\underset{a \ra 0}{\lim} \lambda(a)\cdot v$ exists,
$V_{\lambda,>0}$ is the subset of vectors $v \in V$ such that
$\underset{a \ra 0}{\lim} \lambda(a)\cdot v = 0$,
and $V_{\lambda,0}$ is the subset of vectors $v \in V$ such that
$\underset{a \ra 0}{\lim} \lambda(a)\cdot v = v$.
Furthermore, the limit map
$v\mapsto\underset{a \ra 0}{\lim} \lambda(a)\cdot v$ is nothing
but the projection of $V_{\lambda,\ge0}$ with kernel $V_{\lambda, >0}$
and image $V_{\lambda,0}$.
Of course, similar remarks apply to
$-\lambda$, $V_{\lambda, \le0} := V_{-\lambda, \ge0}$,
and $V_{\lambda,<0}:=V_{-\lambda,>0}$.
If the $G$-module $V$ is defined over $k$,
then each $V_{\lambda,n}$ and $V_{\lambda,>0}$,
etc., is $k$-defined (cf.~\cite[II.5.2]{Bo}).

Now let $T$ be a torus in $G$ with $\lambda\in Y(T)$.
For $\chi\in X(T)$,
let $V_\chi$ denote the corresponding
weight space of $T$ in $V$. If $v\in V$, then we denote
by $v_\chi$ the component of $v$ in the
weight space $V_\chi$ and we put
${\rm supp}_T(v)=\{\chi\in X(T)\mid v_\chi\ne0\}$,
called the \emph{support of $v$ with respect to $T$}.
Then $V_{\lambda, 0}$, $V_{\lambda,\ge0}$ and $V_{\lambda, >0}$
are the direct sums of the subspaces $V_{\lambda,\langle\lambda,\chi\rangle}$,
where $\chi \in X(T)$ is such that $\langle\lambda,\chi\rangle=0$,
$\ge0$ and $>0$,
respectively. Furthermore, $v\in V_{\lambda,\ge0}$ if and only if
$\langle\lambda,\chi\rangle\ge0$ for all $\chi\in{\rm supp}_T(v)$.

Finally, we recall a standard result \cite[Lem.\ 5.2]{Bo1}.
Suppose $T$ is a maximal torus of $G$ with $\lambda\in Y(T)$.
Let $\alpha\in \Psi = \Psi(G,T)$, $v\in V_{\lambda,n}$ and $u\in U_\alpha$.
Then
\begin{equation}
\label{eqn:repborel}
 u\cdot v-v\in \sum_{m\ge1} V_{\lambda,n+ m\langle\lambda,\alpha\rangle}.
\end{equation}
Hence, for any $u\in R_u(P_\lambda)$ and any
$v\in V_{\lambda,\geq 0}$, we have
\begin{equation}
\label{eqn:repborel2}
 u\cdot v-v\in V_{\lambda,>0}.
\end{equation}
\end{rem}

We continue with some further preliminary results used in the proofs below.

\begin{lem}
\label{lem:Ruconj}
Suppose $G$ acts on an affine variety $V$.
Let $v\in V$, let $\lambda\in Y(G)$ and let
$u\in R_u(P_\lambda)$.
Then $\underset{a \ra 0}{\lim}\lambda(a)\cdot v$
exists and equals $u\cdot v$
if and only if $u^{-1}\cdot\lambda$ centralizes $v$.
\end{lem}

\begin{proof}
If $\underset{a \ra 0}{\lim}\lambda(a)\cdot v$
exists and equals $u\cdot v$, then $\lambda$
fixes $u\cdot v$ and therefore $u^{-1}\cdot\lambda$
centralizes $v$. Now assume that the latter is the case.
Then $\underset{a \ra 0}{\lim}\lambda(a)u^{-1}\lambda(a)^{-1}=1$
and $u^{-1}\lambda(a)^{-1}u$ fixes $v$ for all $a\in \ovl{k}^*$, so
$u\cdot v
=
\left(\underset{a \ra 0}{\lim}\lambda(a)u^{-1}\lambda(a)^{-1}\right)\cdot u\cdot v
= \underset{a\ra 0}{\lim}\lambda(a)\cdot (u^{-1}\lambda(a)^{-1}u)\cdot v
= \underset{a\ra 0}{\lim}\lambda(a)\cdot v$.
\end{proof}

\begin{lem}
\label{lem:pconjuconj}
Suppose $G$ acts on an affine variety $V$.
Let $v\in V$, $\lambda\in Y(G)$,
such that $v':=\underset{a\to 0}{\lim} \lambda(a)\cdot v$
exists. Furthermore, let $x\in P_{-\lambda}$ and $u\in R_u(P_\lambda)$
be such that $xu\cdot v$ is $\lambda(\ovl{k}^*)$-fixed. Then
$v'= u\cdot v$.
\end{lem}

\begin{proof}
Without loss, we may assume that $V$ is a rational $G$-module
(cf.\ Remark~\ref{rem:linear}).
Write $x=yl$, where $y\in R_u(P_{-\lambda})$ and $l\in L_\lambda$.
Since $V_{\lambda,\le0}$ is $P_{-\lambda}$-stable and $ylu\cdot v\in V_{\lambda, 0}$,
we have that $lu\cdot v=y^{-1}ylu\cdot v\in V_{\lambda,\le0}$. On the other hand,
$lu\cdot v\in V_{\lambda,\ge0}$, since $v\in V_{\lambda,\ge0}$ and $V_{\lambda,\ge0}$
is $P_\lambda$-stable. So $lu\cdot v\in V_{\lambda, 0}$.
It follows that
$$
lu\cdot v= \underset{a \ra 0}{\lim}\lambda(a)\cdot lu\cdot v=
\underset{a \ra 0}{\lim}\lambda(a) lu\lambda(a)^{-1}\cdot\underset{a \ra 0}{\lim}\lambda(a)\cdot v=
l\cdot v'.
$$
So $v'= u\cdot v$.
\end{proof}

\begin{rem}
The proof of Lemma \ref{lem:pconjuconj}
also works if we replace the assumption that
$xu\cdot v$ is $\lambda(\ovl{k}^*)$-fixed by the weaker assumption
that $\underset{a\to 0}{\lim} \lambda(a)^{-1}\cdot(xu\cdot v)$ exists.
If $x u\cdot v$ is $\lambda(\ovl{k}^*)$-fixed, then we can draw the additional
conclusion that $ylu\cdot v=lu\cdot v$, since $R_u(P_{-\lambda})$
acts trivially on $V_{\lambda,\le0}/V_{\lambda,<0}$,
by Eqn.~\eqref{eqn:repborel2}.
\end{rem}

\begin{lem}
\label{lem:dblecochar}
Let $V$ be a rational $G$-module.
Let $\lambda,\mu\in Y(G)$ such that $\lambda(\ovl{k}^*)$
and $\mu(\ovl{k}^*)$ commute. Then for $t\in {\mathbb N}$
sufficiently large, the following hold:
\begin{enumerate}[{\rm (i)}]
\item $V_{t\lambda+ \mu,\ge0}\subseteq V_{\lambda,\ge0}$,
$V_{t\lambda+ \mu,>0}\supseteq V_{\lambda,>0}$ and
$V_{t\lambda+ \mu,0}=V_{\lambda,0}\cap V_{\mu,0}$;
\item $P_{t\lambda+ \mu}\subseteq P_\lambda$
(hence $R_u(P_{t\lambda+ \mu})\supseteq R_u(P_\lambda)$)
and $L_{t\lambda+ \mu}=L_\lambda\cap L_\mu$.
\end{enumerate}
Furthermore, if $t\in {\mathbb N}$ is such that property (i) holds
and $v\in V$ is such that $v':= \lim_{a\ra 0} \lambda(a)\cdot v$
and $v'':= \lim_{a\ra 0} \mu(a)\cdot v'$ exist,
then $\lim_{a\ra 0} (t\lambda+ \mu)(a)\cdot v$ exists and equals $v''$.
\end{lem}
\begin{proof}
Choose a maximal torus $T$ of $G$ such that $\lambda,\mu\in Y(T)$.
Let $\Phi$ be the set of weights of $T$ on $V$.
Choose $t\in {\mathbb N}$ large enough such that
for any $\chi\in\Phi$ with $\langle \lambda,\chi\rangle\neq 0$,
we have that $\langle t\lambda+\mu,\chi\rangle$ is nonzero and
has the same sign as $\langle \lambda,\chi\rangle$. Then (i) follows.
Part (ii) follows from the argument of the proof of \cite[Prop.~6.7]{martin1} (increasing $t$ if necessary).
Alternatively, it can be deduced from part (i)
by embedding $G$ with the conjugation action
$G$-equivariantly in a
rational $G$-module $W$ and observing that
$P_\nu=W_{\nu,\ge0}\cap G$ and $L_\nu=W_{\nu,0}\cap G$.

Now assume that $t\in {\mathbb N}$ is such that (i) holds
and let $v\in V$ be such that the limits $v'$ and $v''$ above exist.
Since (i) holds, we have for all $\chi\in\Phi$ that
$\langle t\lambda+\mu,\chi\rangle=0$ if and only if
$\langle \lambda,\chi\rangle=0$ and $\langle\mu,\chi\rangle=0$.
For $\nu\in Y(T)$, let $\Phi_{\nu,\geq 0}$ and $\Phi_{\nu,0}$
be the sets of weights $\chi\in\Phi$ such that
$\langle \nu,\chi\rangle\ge0$ and $\langle \nu,\chi\rangle=0$, respectively.
Then ${\rm supp}_T(v)\subseteq \Phi_{\lambda,\geq 0}$ and
${\rm supp}_T(v)\cap \Phi_{\lambda,0}\subseteq \Phi_{\mu,\geq 0}$.
It follows that $\lim_{a\ra 0} (t\lambda+ \mu)(a)\cdot v$ exists
and equals $v''$.
\end{proof}

We finish the section with a result that lets us pass from $k$-points
to arbitrary points. Let $V$ be a $k$-defined
rational $G$-module and let $k_1/k$ be a field extension.
Let $v\in V(k_1)$ and let $\lambda\in Y_k(G)$.
Pick a basis $(\alpha_i)_{i\in I}$ for $k_1$
over $k$; then we can write
$v = \sum_{i \in J} \alpha_iv_i$ for some finite subset $J$ of $I$
and certain (unique) $v_i \in V(k)$.
Clearly, we may assume that $J=\{1,\ldots,n\}$ for some $n\in {\mathbb N}$.
Set $\mathbf{v} = (v_1,\ldots,v_n) \in V^n$ and let
$G$ act diagonally on $V^n$.

\begin{lem}
\label{lem:notrat}
With the notation as above, the following hold:
\begin{enumerate}[{\rm(i)}]
\item $\lim_{a\ra 0} \lambda(a)\cdot v$ exists if and only if $\lim_{a\ra 0} \lambda(a)\cdot {\mathbf v}$ exists if and only if $\lim_{a\ra 0} \lambda(a)\cdot v_i$ exists for each $i$.
\item Suppose the limits in (i) exist. Then for any $g\in G(k)$, we have $v'= g\cdot v$ if and only if ${\mathbf v}'= g\cdot {\mathbf v}$ if and only if $v_i'= g\cdot v_i$ for each $i$.
\end{enumerate}
\end{lem}

\begin{proof}
Part (ii) is obvious.  In part (i), it follows easily from the definitions of limit and direct product that the second limit exists if and only if the third limit exists. Since $\lambda$ is $k$-defined, $V_{\lambda,\ge 0}$ is $k$-defined, so $V_{\lambda,\ge 0}=\bigoplus_{i\in I}\alpha_i(V_{\lambda,\ge 0}\cap V(k))$. So $v\in V_{\lambda,\ge 0}$ if and only if $v_i\in V_{\lambda,\ge 0}$ for all $i\in J$. Hence the first limit exists if and only if the third limit exists.  This completes the proof.
\end{proof}

\section{Orbits and rationality}
\label{sec:orbrat}

In this section we prove some results about $G(k)$-orbits as indicated in the
Introduction.  We maintain the notation from the previous sections; recall in particular that $\Gamma = \Gal(k_s/k)$.

Suppose $V$ is a $k$-defined affine $G$-variety. Even when one is interested mainly in rationality questions, one must sometimes consider points $v\in V$ that are not $k$-points.  For instance, we often want to prove results about a $k$-defined subgroup $H$ of $G$ by choosing a generating tuple $\tuple{h}= (h_1,\ldots, h_n)\in H^n$ for some $n\in {\mathbb N}$, but $H$ need not admit such a tuple with the $h_i$ all being $k$-points (for example when $k$ is finite and $H$ is infinite).
Fortunately, the weaker property that $C_{G(k_s)}(v)$ is $\Gamma$-stable will often suffice (see Theorem \ref{thm:sepRuconj}, for example). We do not require $C_G(v)$ to be $k$-defined here; note that even when $v$ is a $k$-point, $C_G(v)$ is $k$-closed but need not be $k$-defined.  Some of our results hold without any rationality assumptions on $v$ at all (see Theorems \ref{thm:Ruconj} and \ref{thm:cocharclosedcrit}).

\begin{thm}
\label{thm:sepRuconj}
Suppose $V$ is a $k$-defined affine $G$-variety.
Let $v \in V$ and let $\lambda \in Y_k(G)$ be such that
$v':=\lim_{a\to 0}\lambda(a)\cdot v$ exists.
If $v'$ is $R_u(P_\lambda)(k_s)$-conjugate to $v$ and $C_{G(k_s)}(v)$ is $\Gamma$-stable,
then $v'$ is $R_u(P_\lambda)(k)$-conjugate to $v$.
\end{thm}

\begin{proof}
Since $C_{G^0(k_s)}(v)=C_{G(k_s)}(v)\cap G^0(k_s)$, we may assume that $G$ is connected.
Set $P = P_\lambda$.
By hypothesis, there exists $u \in R_u(P)(k_s)$ such that $v' = u \cdot v$.
By Lemma~\ref{lem:Ruconj}, $\mu := u\inverse\cdot\lambda \in Y_{k_s}(G)$
centralizes $v$,
so $\mu(\ovl{k}^*) \subseteq C_P(v)$,
and $\mu(k_s^*) \subseteq C_{G(k_s)}(v) \cap P$.
Note that since $u \in P$, we have $P_\mu = P$.
Let $H$ be the subgroup of $G$ generated by the $\Gamma$-conjugates of
$\mu(\ovl{k}^*)$;
then the union of the $\Gamma$-conjugates of $\mu(k_s^*)$ is dense in $H$, so $H$ is closed, connected and $k_s$-defined,
by \cite[AG.14.5, I.2.2]{Bo},
and $H \subseteq P$, since $\mu(\ovl{k}^*) \subseteq P$
and $P$ is $\Gamma$-stable.
Moreover, since $C_{G(k_s)}(v)$ is $\Gamma$-stable, we can conclude that $H \subseteq C_P(v)$.
Since $H$ has a $\Gamma$-stable dense set of separable points, $H$ is $k$-defined,
and hence contains a $k$-defined maximal torus $S$.
There exists $h \in H$ such that $\mu' := h \cdot\mu$ belongs to $Y(S)$;
note that $\mu'$ centralizes $v$, and since $h \in P$,
we deduce that $P_{\mu'} = P$. In case $k$ is perfect, $C_P(v)$ is $k$-defined, since $C_P(v)=C_G(v)\cap P$ is $\Gamma$-stable.
So in this case we could simply have taken $S$ to be a $k$-defined maximal torus of $C_P(v)$.

By \cite[Cor.~III.9.2]{Bo}, $C_P(S)$ is $k$-defined,
so it has a $k$-defined maximal torus $T$.
Note that $S\subseteq T$, since $S$ commutes
with $T$ and $T$ is maximal.
There exists a unique $k$-defined Levi subgroup $L$ of $P$ containing $T$, by Lemma \ref{lem:Rparrat}(iii).
But $L_{\mu'}$ is an Levi subgroup of $P$ containing $T$, so $L_{\mu'} = L$.
Thus we have two Levi subgroups $L_\lambda$ and $L_{\mu'}$ of $P$,
both $k$-defined. By Lemma \ref{lem:Rparrat}(iii),
there exists a unique $u_0 \in R_u(P)(k)$ such that
$L_\lambda = u_0 L_{\mu'} u_0\inverse$.
We also have $\mu' = hu\inverse \cdot \lambda$, and since $hu\inverse \in P$,
we can write $hu\inverse = u_1l$ with $u_1 \in R_u(P)$ and $l \in L_\lambda$.
But $L_\lambda$ centralizes $\lambda$,
so $\mu' = u_1l\cdot\lambda = u_1\cdot\lambda$.
So $u_0\inverse L_\lambda u_0 = L_{\mu'}=L_{u_1\cdot\lambda}
= u_1L_\lambda u_1\inverse.$
Since $R_u(P)$ acts simply transitively on the set of Levi subgroups of $P$,
we must have $u_1 = u_0\inverse$, and hence $\mu' = u_0\inverse\cdot\lambda$.
Applying Lemma \ref{lem:Ruconj} again, we see that
$v' = u_0\cdot v$, because $\mu'$ centralizes $v$.
This proves the theorem.
\end{proof}

\begin{exmp}
\label{exmp:wkdefnec}
 The assumption that $C_{G(k_s)}(v)$ is $\Gamma$-stable in Theorem \ref{thm:sepRuconj}
is necessary.  For instance, let $G= \SL_2$ act on $V=G$ by conjugation.
Choose $y\in k_s\setminus k$ and $x\in k^* \setminus \{\pm 1\}$.
Let $v'= \twobytwo{x}{0}{0}{x^{-1}}$ and
$v= \twobytwo{1}{y}{0}{1} \twobytwo{x}{0}{0}{x^{-1}} \twobytwo{1}{-y}{0}{1}$,
and define $\lambda\in Y_k(G)$ by $\lambda(a)= \twobytwo{a}{0}{0}{a^{-1}}$.
It is easily seen that $v'= \lim_{a\ra 0} \lambda(a)\cdot v$ and that $v'$
is $R_u(P_\lambda)(k_s)$-conjugate to $v$ but not
$R_u(P_\lambda)(k)$-conjugate to $v$.
\end{exmp}

We can now state our first main result.

\begin{thm}
\label{thm:Ruconj}
Suppose $k$ is perfect.
Suppose $V$ is a $k$-defined affine $G$-variety and let $v\in V$.
Let $\lambda\in Y_k(G)$ such
that $v':=\underset{a\to 0}{\lim} \lambda(a)\cdot v$
exists and is $G(k)$-conjugate to $v$.
Then $v'$ is $R_u(P_\lambda)(k)$-conjugate to $v$.
\end{thm}

\begin{proof}
Fix a maximal torus $T$ of $P_\lambda$
such that $\lambda\in Y(T)$ and
a Borel subgroup $B$ of $P_\lambda^0$ such
that $T\subseteq B$.  Let $B^-$ be the
Borel subgroup of $G$ opposite to $B$ with
respect to $T$; note that $B^-\subseteq P_{-\lambda}$.

We begin with the case that $k$ is algebraically closed.
We can assume that $v\neq v'$.
For $a\in k^*$, set $v_a= \lambda(a)\cdot v$;
then $v_a\ne v'$ for all $a\in k^*$.
We show that $v\in R_u(P_\lambda)P_{-\lambda}^0 \cdot v'$.
Let $\varphi:G\to G\cdot v'$ be the orbit map of $v'$.
Then $\varphi$ is open, by \cite[AG Cor.~18.4]{Bo}.
The set $R_u(P_\lambda)P_{-\lambda} ^0$
contains the big cell $BB^-\subseteq G^0$, which
is an open neighbourhood of $1$ in $G$,
so $\varphi(R_u(P_\lambda)P_{-\lambda}^0)$ contains
an open neighbourhood of $\varphi(1)=v'$ in $G\cdot v'$.
The image of $k^*$ under the limit morphism
$\widehat\phi_{v,\lambda}:k \to V$ meets
this neighbourhood,
so there exists $a\in k^*$ such that $v_a\in R_u(P_\lambda)P_{-\lambda}^0 \cdot v'$.
But $R_u(P_\lambda)$ is normal
in $P_\lambda$ and hence $R_u(P_\lambda)P_{-\lambda}^0 $ is stable under left multiplication by elements of $T$,
so we in fact have $v\in R_u(P_\lambda)P_{-\lambda}^0 \cdot v'$.
Lemmas \ref{lem:Rparrat}
and \ref{lem:pconjuconj} now imply that
$v'= u\cdot v$ for some $u\in R_u(P_\lambda)$.
This completes the proof when $k$ is algebraically closed.

Now assume $k$ is perfect. First assume $v\in V(k)$.  By the algebraically closed case, we know that $v$ and $v'$ are
$R_u(P_\lambda)$-conjugate. Since $C_{G(k_s)}(v)$ is $\Gamma$-stable, we can apply Theorem \ref{thm:sepRuconj} to deduce that $v$ and $v'$ are $R_u(P_\lambda)(k)$-conjugate.  Now let $v$ be arbitrary.
There is no loss in assuming that $V$ is a $k$-defined rational $G$-module
(cf.\ Remark \ref{rem:linear}).
Let ${\mathbf v}, {\mathbf v}'\in V^n$ be as in Lemma \ref{lem:notrat}.  Then $\lim_{a\ra 0}\lambda(a)\cdot {\mathbf v}= {\mathbf v}'$ and ${\mathbf v}$ and ${\mathbf v}'$ are $G(k)$-conjugate, so they are $R_u(P_\lambda)(k)$-conjugate by the argument above.  Lemma~\ref{lem:notrat}(ii)
implies that $v$ and $v'$ are $R_u(P_\lambda)(k)$-conjugate, as required.
\end{proof}

\begin{rem}
\label{rem:Ruconj}
Theorem \ref{thm:Ruconj} was first proved by H.~Kraft and J.~Kuttler
for $k$ algebraically closed of characteristic zero
in case $V = G/H$ is an affine homogeneous space
(by a method different from ours), cf.\
\cite[Prop.\ 2.1.4]{schmitt} or \cite[Prop.\ 2.1.2]{Gomez}.
We do not know whether this theorem holds for arbitrary $k$.
\end{rem}

The following consequence of Theorem \ref{thm:Ruconj} is used in the proof of \cite[Prop.~3.34]{BMRT:relative}.

\begin{cor}
\label{cor:doubleorbit}
Let $G_1$ and $G_2$ be reductive groups
and let $V$ be an affine $(G_1\times G_2)$-variety.  Let $v\in V$ and $\lambda_1\in Y(G_1)$
and assume that $v':=\lim_{a\ra 0}\lambda_1(a)\cdot v$ exists.
Then the following hold:
\begin{enumerate}[{\rm (i)}]
\item If $v'$ is $(G_1\times G_2)$-conjugate to $v$, then it is $G_1$-conjugate to $v$.
In particular, $G_1\cdot v$ is closed if $(G_1\times G_2)\cdot v$ is.
\item Let $\pi:V\to V/\!\!/G_2$ be the canonical projection and assume that
$\pi^{-1}(\pi(v))=G_2\cdot v$. If $\pi(v')$ is
$G_1$-conjugate to $\pi(v)$, then $v'$ is $G_1$-conjugate to $v$.
\end{enumerate}
\end{cor}

\begin{proof}
(i). By Theorem~\ref{thm:Ruconj}, there exists $u\in R_u(P_{\lambda_1}(G_1\times G_2))$ such that
$v'= u\cdot v$.  But $R_u(P_{\lambda_1}(G_1\times G_2))= R_u(P_{\lambda_1}(G_1))\times \{1\}$, so $v'$
is $G_1$-conjugate to $v$, as required.  The second assertion follows immediately from
the Hilbert-Mumford Theorem.

(ii). This follows immediately from (i).
\end{proof}

The following example shows that
the converse of Corollary \ref{cor:doubleorbit}(i) does not
hold in general.

\begin{exmp}
Let $G = G_1 \times G_2$, where $G_i= k^*$ for $i = 1,2$
(here $k$ is assumed algebraically closed).
Set $V = k^2$, and let $G$ act on $V$ as follows:
$$
(t_1,t_2)\cdot(x_1,x_2) := (t_1^2t_2\inverse x_1, t_2^2t_1\inverse x_2),
$$
for $t_i \in G_i$ and $(x_1,x_2) \in V$.
Consider the point $(1,1) \in V$.
Then the $G_i$-orbits of $(1,1)$ are clearly closed,
but the $G$-orbit of $(1,1)$ is not closed (if $\lambda \in Y(G)$ is given by $\lambda(a) = (a,a)$,
then $\lim_{a\ra 0} \lambda(a)\cdot (1,1)= (0,0)$).
\end{exmp}

Here is a further consequence of Theorem \ref{thm:Ruconj}:
it gives a criterion for determining whether an orbit
is closed when $k$ is perfect.

\begin{cor}
\label{cor:perfectclsd}
Assume $k$ is perfect.
Let $V$ be a $k$-defined affine $G$-variety and let $v\in V$ such that $C_{G(k_s)}(v)$ is $\Gamma$-stable.
Suppose there exists $\lambda\in Y_k(G)$ such that
$v':= \lim_{a\ra 0} \lambda(a)\cdot v$ exists and
is not $R_u(P_\lambda)(k)$-conjugate to $v$.
Then $v'\not\in G\cdot v$.  In particular, $G\cdot v$ is not closed.
\end{cor}

\begin{proof}
 Suppose $v'$ is $G$-conjugate to $v$.
Then $v'$ is $R_u(P_\lambda)$-conjugate to $v$, by
Theorem \ref{thm:Ruconj}.  Since $C_{G(k_s)}(v)$ is $\Gamma$-stable and $k$ is perfect,
$v'$ is $R_u(P_\lambda)(k)$-conjugate to $v$ by
Theorem \ref{thm:sepRuconj},
a contradiction.  Hence $v'\not\in G\cdot v$,
and thus this orbit is not closed.
\end{proof}

In order to state our next main result,
we need an appropriate extension of the concept of orbit closure
to the non-algebraically closed case.

\begin{defn}
\label{def:cocharclosure}
Let $V$ be a $k$-defined affine
$G$-variety.  Let $v\in V$.  We say that the
$G(k)$-orbit $G(k)\cdot v$ is \emph{cocharacter-closed over $k$}
if for any $\lambda\in Y_k(G)$
such that $v':= \lim_{a\ra 0} \lambda(a)\cdot v$ exists,
$v'$ is $G(k)$-conjugate to $v$.
Note that we do not require $v$ to be a $k$-point of $V$.
\end{defn}

\begin{rem}
\label{rem:cocharclosure}
In what follows we give $V(k)$ the topology induced by the Zariski topology~of~$V$.

(i). Let $v\in V(k)$. If $k$ is infinite, then $G(k)$ is dense in $G$ \cite[V.18.3~Cor.]{Bo}, so $G(k)\cdot v$ is dense in the closure of $G\cdot v$.  It follows easily that if $G(k)\cdot v$ is closed in $V(k)$, then $G(k)\cdot v$ is cocharacter-closed over $k$. Corollary~\ref{cor:perfectorbitcrit} below
now implies that if $k$ is infinite and perfect and $G(k)\cdot v$ is closed in $V(k)$, then $G\cdot v$ is closed. On the other hand, if $k$ is finite, then $G(k)\cdot v$ is a finite subset of $V(k)$ and hence is closed in $V(k)$, even though $G(k)\cdot v$ need not be cocharacter-closed over $k$.

(ii). If $G(k)\cdot v$ is cocharacter-closed over $k$, then $G(k)\cdot v$ need not be closed in $V(k)$, even if $G\cdot v$ is closed.  We give two examples.  First, let $k$ be a non-perfect field of characteristic $p>0$. Let $G=\ovl{k}^*$ acting on $V=\ovl{k}^*$ by $g\cdot v= g^pv$. Give $G$ and $V$ the obvious $k$-structures. Then $G(k)\cdot 1=(k^*)^p$ is not closed in $V(k)=k^*$, but $G\cdot 1=\ovl{k}^*$ is closed. Moreover, $G(k)\cdot 1$ is cocharacter-closed over $k$, since the limit $\lim_{a\ra 0} \lambda(a)\cdot 1$ does not exist for any non-trivial $\lambda\in Y(G)$.

Second, let $k={\mathbb R}$ and let $G= {\rm SL}_2$ acting on $V=G$ by conjugation.  Let $v= \twobytwo{0}{-1}{1}{0}$, $w= \twobytwo{0}{1}{-1}{0}$, and $g= \twobytwo{i}{0}{0}{-i}$.  Then $v,w\in V({\mathbb R})$ and $w= g\cdot v$, so $v$ and $w$ are $G({\mathbb C})$-conjugate, but it is easily checked that they are not $G({\mathbb R})$-conjugate.  Hence $w$ lies in the closure in $V({\mathbb R})$ of $G({\mathbb R})\cdot v$, which implies that $G({\mathbb R})\cdot v$ is not closed in $V({\mathbb R})$.  But $G\cdot v$ is closed since $v$ is semisimple (\cite[III.9.2~Thm.]{Bo}), so $G({\mathbb R})\cdot v$ is cocharacter-closed over ${\mathbb R}$ by Corollary \ref{cor:perfectclsd}.

(iii). In \cite{levy}, Levy investigated
a notion similar to our concept of cocharacter-closure
in case of a rational $G$-module in characteristic $0$.
\end{rem}

Our next result says that we can remove the hypothesis
that $k$ is perfect in Theorem \ref{thm:Ruconj}
if we assume that $G$ is connected and $G(k)\cdot v$ is cocharacter-closed over $k$.

\begin{thm}
\label{thm:cocharclosedcrit}
Suppose $V$ is a $k$-defined affine $G$-variety.
Assume that $G$ is connected. Let $v\in V$.
Then the following are equivalent:
\begin{enumerate}[{\rm (i)}]
\item $G(k)\cdot v$ is cocharacter-closed over $k$;
\item for all $\lambda\in Y_k(G)$,
if $v':= \lim_{a\ra 0} \lambda(a)\cdot v$ exists,
then $v'$ is $R_u(P_\lambda)(k)$-conjugate to $v$.
\end{enumerate}
\end{thm}

\begin{proof}
It is immediate that (ii) implies (i), so we need to prove that
(i) implies (ii).
Assume $G(k)\cdot v$ is cocharacter-closed over $k$.
Without loss of generality we can assume that $V$ is
a $k$-defined rational $G$-module (cf.\ Remark~\ref{rem:linear}).
We argue by induction on $\dim V_{\lambda,0}$
for $\lambda\in Y_k(G)$.
Suppose $\lambda\in Y_k(G)$ and let $v\in V$ such that
$v':= \lim_{a\ra 0} \lambda(a)\cdot v$ exists.
If $\dim V_{\lambda,0}=0$, then $V_{\lambda,0}=0$ and so $v'=0$, which forces $v=0$ and we are done.
Let $S$ be a maximal $k$-split torus of $G$ with $\lambda\in Y_k(S)$, let
${}_k\Psi$ be the set of roots of $G$ relative to $S$ and
let ${}_kW=N_G(S)/C_G(S)$ be the Weyl group over $k$.
Any $w\in {}_kW$ has a representative
in $N_G(S)(k)$, see \cite[V.21.2]{Bo}.
We have $C_G(S)\subseteq P_\lambda$.
Fix a minimal
$k$-defined parabolic subgroup $P$ of $G$
with $C_G(S)\subseteq P\subseteq P_\lambda$.
Using the notation of \cite[V.21.11]{Bo},
the choice of $P$ corresponds to a choice of
simple roots ${}_k\Delta\subseteq{}_k\Psi$ and then,
$P_\lambda={}_kP_J$ for a unique subset $J$ of ${}_k\Delta$.
Define the subset ${}_kW^J$ of ${}_kW$ as in \cite[V.21.21]{Bo},
and for each $w\in{}_kW^J$ define the subgroup $U_w'$ of $R_u(P_\lambda)$,
as in \cite[V.21.14]{Bo}. For each $w\in{}_kW^J$, let $\dot{w}$ be
a representative of $w$ in $N_G(S)(k)$.
Then, by \cite[V.21.16 and V.21.29]{Bo}
or \cite[3.16 proof]{BoTi},  we have
\begin{equation}
\label{eqn:bruhat}
G(k)=\bigcup_{w\in{}_kW^J}U_w'(k)\dot{w}P_\lambda(k)
\end{equation} and
\begin{equation}
\label{eqn:conjugate}
\dot{w}^{-1}U_w'\dot{w}\subseteq R_u(P_{-\lambda})\text{\quad for each }w\in{}_kW^J.
\end{equation}

Since $G(k)\cdot v$ is cocharacter-closed over $k$,
there exists $g\in G(k)$ such that $v'= g\cdot v$.
By \eqref{eqn:bruhat} and Lemma~\ref{lem:Rparrat},
we have $g=u'\dot{w}lu$ for some $w\in{}_kW^J$,
$u'\in U_w'(k)$, $l\in L_\lambda(k)$ and $u\in R_u(P_\lambda)(k)$.
Now the argument splits in two cases. Put $n=\dot{w}$.\smallskip\\
{\bf Case 1:} $n$ normalizes $V_{\lambda,0}$.
Then $n^{-1}u'nlu\cdot v=n^{-1}\cdot v' \in V_{\lambda,0}$.
Furthermore, $n^{-1}u'n\in R_u(P_{-\lambda})$ by \eqref{eqn:conjugate}, so $n^{-1}u'nl\in P_{-\lambda}$.
The desired conclusion follows from
Lemma~\ref{lem:pconjuconj}.\smallskip\\
{\bf Case 2:} $n$ does not normalize $V_{\lambda,0}$.
Let $\Phi$ be the set of weights of $S$ on $V$
and for $\nu\in Y_k(S)$ let $\Phi_{\nu,\geq 0}$ be the set
of weights $\chi\in\Phi$ such that $\langle \nu,\chi\rangle\ge0$.
We have $v'=u'nlu\cdot v$ and therefore $n^{-1}u'^{-1}\cdot v'=lu\cdot v$.
Furthermore, $u'^{-1}\cdot v'- v'\in V_{\lambda,>0}$
by  ~\eqref{eqn:repborel2},
whence ${\rm supp}_S(v')\subseteq {\rm supp}_S(u'^{-1}\cdot v')$.
Now $n^{-1}$ normalizes $S$, so
$$
n^{-1}\cdot{\rm supp}_S(v')\subseteq n^{-1}\cdot{\rm supp}_S(u'^{-1}\cdot v')
={\rm supp}_S(n^{-1}u'^{-1}\cdot v')\subseteq \Phi_{\lambda,\geq 0},
$$
since $n^{-1}u'^{-1}\cdot v'=lu\cdot v\in V_{\lambda,\ge0}$.
It follows that
${\rm supp}_S(v')\subseteq n\cdot\Phi_{\lambda,\geq 0}=\Phi_{n\cdot\lambda,\geq 0}$.
So $v'':= \lim_{a\ra 0} (n\cdot \lambda)(a)\cdot v'$ exists.
We can choose $\gamma\in Y_k(G)$
of the form $\gamma= t\lambda+ n\cdot \lambda$ for $t\in {\mathbb N}$
sufficiently large such that the following hold:
\begin{enumerate}[(1)]
\item $V_{\gamma,0}\subseteq V_{\lambda,0}$ and $V_{\gamma,0}\subseteq V_{n\cdot \lambda,0}$;
\item $v''= \lim_{a\ra 0} \gamma(a)\cdot v$;
\item $v''= \lim_{a\ra 0} \gamma(a)\cdot v'$;
\item $P_\gamma\subseteq P_\lambda$.
\end{enumerate}
Properties (1), (2) and (4) follow immediately from Lemma \ref{lem:dblecochar}, while (3) follows from (2), since
$\lim_{a\ra 0} (n\cdot \lambda)(a)\cdot v'=v''$ and
$\lambda(k^*)$ fixes $v'$.
If $V_{\lambda,0}=V_{\gamma,0}$, then $V_{\lambda,0}\subseteq V_{n\cdot \lambda,0}=n\cdot V_{\lambda,0}$, so $V_{\lambda,0}=n\cdot V_{\lambda,0}$,
contradicting the fact that $n$ does not normalize $V_{\lambda,0}$.
Hence we must have $\dim V_{\gamma,0}< \dim V_{\lambda,0}$.
Now $G(k)\cdot v'=G(k)\cdot v$ is cocharacter-closed over $k$.
So, by the induction hypothesis, $v$ and $v'$
are both $R_u(P_\gamma)(k)$-conjugate --- and hence $P_\gamma(k)$-conjugate --- to $v''$,
and hence $v$ and $v'$ are $P_\gamma(k)$-conjugate.
By (4), $v$ and $v'$ are $P_\lambda(k)$-conjugate.
But then they are $R_u(P_\lambda)(k)$-conjugate,
by Lemmas~\ref{lem:Rparrat} and \ref{lem:pconjuconj}.
\end{proof}

In view of Theorems \ref{thm:sepRuconj}, \ref{thm:Ruconj} and \ref{thm:cocharclosedcrit},
it is natural to ask the following rationality question.

\begin{question}
\label{qn:cocharclsdext}
Let $V$ be a $k$-defined affine
$G$-variety.  Let $v\in V$ such that $C_{G(k_s)}(v)$ is $\Gamma$-stable.
Suppose $k_1/k$ is an algebraic extension.
Is it true that $G(k_1)\cdot v$ is cocharacter-closed over $k_1$
if and only if $G(k)\cdot v$ is cocharacter-closed over $k$?
\end{question}

Our final result in this section gives an affirmative answer to the forward implication of
Question \ref{qn:cocharclsdext}
in two instances.

\begin{thm}
\label{thm:GeneralSeparableDown}
Let $k_1/k$ be an algebraic extension of fields and
let $V$ be a $k$-defined affine $G$-variety.
Let $v \in V$ such that $C_{G(k_s)}(v)$ is $\Gamma$-stable.
Suppose that (i) $G$ is connected and $k_1/k$ is separable, or (ii) $k$ is perfect.
If $G(k_1)\cdot v$ is cocharacter-closed over $k_1$,
then $G(k)\cdot v$ is cocharacter-closed over $k$.
\end{thm}

\begin{proof}
Suppose $\lambda \in Y_k(G)$ such that $v' = \lim_{a\ra 0} \lambda(a)\cdot v$
exists.  Then $\lambda \in Y_{k_1}(G)$.  Since
$G(k_1)\cdot v$ is cocharacter-closed over $k_1$,
$v'$ is $R_u(P_\lambda)(k_1)$-conjugate to $v$, by
Theorem \ref{thm:cocharclosedcrit} in case (i)
and Theorem \ref{thm:Ruconj} in case (ii).
Since $k_1/k$ is separable,  Theorem \ref{thm:sepRuconj}
implies that $v'$ is $R_u(P_\lambda)(k)$-conjugate to $v$.
Hence $G(k)\cdot v$ is cocharacter-closed over $k$, as required.
\end{proof}

\begin{rems}
(i). If $v\in V(k)$, then the reverse direction holds for $k$ perfect in Theorem \ref{thm:GeneralSeparableDown} and the answer to Question \ref{qn:cocharclsdext} is yes: this follows from Corollary \ref{cor:perfectorbitcrit} below.

(ii). For arbitrary $k$ it can happen that $G(k)\cdot v$ is not
cocharacter-closed over $k$ but $G\cdot v$ is closed or vice versa, even when $v\in V(k)$ (see Remark \ref{rem:obstacle}; cf.\ also Remark \ref{rems:optimal}(ii)).
\end{rems}

\section{Uniform $S$-instability}
\label{sec:uniform}

In this section we show that the results of Kempf in \cite{kempf}
extend to \emph{uniform instability} as defined by W.~Hesselink in \cite{He}.
Since this is a straightforward modification of Kempf's arguments, we
only indicate the relevant changes.
We point out here that
the extension to non-connected $G$ is unnecessary for
the results in this section,
since they follow immediately from the corresponding
statements for $G$ connected.
We state the results for $G$
non-connected, because this is more convenient for our
applications.
As our field $k$ is not necessarily algebraically closed,
we restrict to
$k$-defined cocharacters of $G$ in Kempf's optimization procedure,
cf. \cite{He}.

Throughout this section,
$G$ is a reductive $k$-defined normal subgroup of a $k$-defined linear algebraic group $G'$
which acts on an
affine $\ovl{k}$-variety
$V$, and $S$ is a non-empty $G$-stable closed subvariety of $V$.

\begin{defn}
\label{def:norm}
A \emph{$G'(k)$-invariant norm} on $Y_k(G)$ is a
non-negative real-valued function $\left\|\,\right\|$ on $Y_k(G)$
such that
\begin{itemize}
\item[(i)] $\left\| g \cdot \lambda \right\| = \left\| \lambda \right\|$
for any $g \in G'(k)$ and any $\lambda \in Y_k(G)$,
\item[(ii)] for any $k$-split $k$-defined torus $T$ of $G$, there is a positive definite
integer-valued form $(\ {,}\ )$ on $Y_k(T)$ such that
$(\lambda, \lambda) = \left\| \lambda \right\|^2$ for any $\lambda \in Y_k(T)$.
\end{itemize}

If $k=\ovl{k}$, then we speak of a \emph{$G'$-invariant norm} on $Y(G)$, and in this case we say that $\left\|\,\right\|$ on $Y(G)$ is \emph{$k$-defined}
if it is $\Gamma$-invariant (see \cite[Sec.\ 4]{kempf}).  Note that a $G'$-invariant norm $\left\|\,\right\|$ on $Y(G)$ determines
a $G'(k)$-invariant norm on $Y_k(G)$.  A $k$-defined $G'$-invariant norm on $Y(G)$ always exists,
by the argument of \cite[1.4]{He}.
\end{defn}

\begin{defn}
\label{def:destabilizing}
For each non-empty subset $X$ of $V$,
define $\Lambda(X)$ as the set of
$\lambda \in Y(G)$ such that $\underset{a\to 0}{\lim}\, \lambda(a)\cdot x$
exists for all $x\in X$, and put
$\Lambda(X,k)=\Lambda(X)\cap Y_k(G)$.
Extending Hesselink \cite{He}, we call $X$ \emph{uniformly $S$-unstable}
if there exists
$\lambda \in \Lambda(X)$ such that
$\underset{a\to 0}{\lim}\, \lambda(a)\cdot x \in S$ for all $x \in X$,
and we say that such a cocharacter \emph{destabilizes $X$ into $S$}
or is a \emph{destabilizing cocharacter for $X$ with respect to $S$}.
We call $X$ \emph{uniformly $S$-unstable over $k$} if there exists
such a $\lambda$ in $\Lambda(X,k)$.
We say that $x\in V$ is \emph{$S$-unstable over $k$} if
$\{x\}$ is uniformly $S$-unstable over $k$.
Finally, (uniformly $S$-) unstable without
specifying a field  always means (uniformly $S$-)
unstable over $\ovl{k}$. By the
Hilbert-Mumford Theorem, $x\in V$ is $S$-unstable
if and only if $\ovl{G\cdot x}\cap S\ne\varnothing$.
\end{defn}

\begin{rem}
Following Hesselink \cite[(2.1)]{He}, we allow the trivial case that $X\subseteq S$.
In this case the optimal class of Definition~\ref{def:optimalclass} below consists just of the trivial cocharacter $\lambda=0$ and the optimal destabilizing parabolic subgroup of Definition~\ref{def:optimalparabolic} is the whole of $G$.
Kempf \cite[Thm.~3.4]{kempf} only defines the optimal class and optimal destabilizing parabolic subgroup if $X=\{x\}$ and $x\notin S$ and in this case our definitions coincide with his.
\end{rem}

Let $x\in V$ and let $\lambda\in\Lambda(x)$.
Let $\varphi:\ovl{k}\to V$ be the morphism
$\widehat\phi_{x,\lambda}$ from Section~\ref{subsec:Gvars}.
If $x\notin S$, then the scheme-theoretic inverse image
$\varphi^{-1}(S)$ is either empty or has affine ring
$\ovl{k}[T]/(T^m)$ for a unique $m\in {\mathbb N}$, and we define
$a_{S,x}(\lambda):=m$ (taking $m$ to be 0 if $\varphi^{-1}(S)$ is empty).
If $x\in S$, then we define $a_{S,x}(\lambda):=\infty$.
For a non-empty subset $X$ of $V$ and $\lambda\in\Lambda(X)$,
we define $a_{S,X}(\lambda):=\min_{x\in X}a_{S,x}(\lambda)$. Note that
$a_{S,X}(\lambda)>0$ if and only if
$\underset{a\to 0}{\lim}\, \lambda(a)\cdot x\in S$
for all $x\in X$, $a_{S,X}(\lambda) = 0$ if and only if
$\underset{a\to 0}{\lim}\, \lambda(a)\cdot x\not\in S$ for some $x \in X$,
and $a_{S,X}(\lambda)=\infty$ if and only if $X\subseteq S$.

Now we choose a $G'(k)$-invariant norm $\left\|\,\right\|$ on $Y_k(G)$.

\begin{defn}
\label{def:optimalclass}
Let $X$ be a non-empty subset of $V$. If $X\subseteq S$,
we put $\Omega(X,S,k)=\{0\}$,
where $0$ denotes the trivial cocharacter of $G$.
Now assume $X\nsubseteq S$. If the
function $\lambda\mapsto a_{S,X}(\lambda)/{\left\|\lambda\right\|}$
has a finite
strictly positive maximum value on $\Lambda(X,k)\setminus\{0\}$, then we define
$\Omega(X,S,k)$ as the set of indivisible cocharacters in
$\Lambda(X,k)\setminus\{0\}$ on
which this function takes its maximum value. Otherwise we define
$\Omega(X,S,k)=\varnothing$.
Note that $X$ is uniformly $S$-unstable over $k$
(in the sense of Definition \ref{def:destabilizing}) provided
$\Omega(X,S,k)\ne\varnothing$. The set  $\Omega(X,S,k)$ is called the
\emph{optimal class for $X$ with respect to $S$ over $k$}.
\end{defn}

We are now able to state and prove the analogue of Kempf's
instability theorem (\cite[Thm.\ 4.2]{kempf}) in this setting.

\begin{thm}
\label{thm:kempfrousseau}
Let $X$ be a non-empty subset of $V$ which is uniformly $S$-unstable over $k$.
Then $\Omega(X,S,k)$ is non-empty and has the following properties:
\begin{enumerate}[{\rm(i)}]
\item $\underset{a\to 0}{\lim}\, \lambda(a)\cdot x \in S$ for all
$\lambda \in \Omega(X,S,k)$ and any $x \in X$.
\item For all $\lambda, \mu \in \Omega(X,S,k)$, we have $P_\lambda = P_\mu$.
Let $P(X,S,k)$ denote the unique R-parabolic subgroup of $G$ so defined.  (Note that $P(X,S,k)$ is $k$-defined by Lemma \ref{lem:Rparrat}.)
\item If $g \in G'(k)$, then
$\Omega(g\cdot X,g\cdot S,k) = g\cdot \Omega(X,S,k)$
and $P(g\cdot X,g\cdot S,k) = gP(X,S,k)g^{-1}$.
\item $R_u(P(X,S,k))(k)$ acts simply transitively on $\Omega(X,S,k)$: that is,
for each $k$-defined R-Levi subgroup $L$ of $P(X,S,k)$, there exists one
and only one $\lambda \in \Omega(X,S,k)$ with $L = L_\lambda$.
Moreover, $N_{G(k)}(X) \subseteq P(X,S,k)(k)$.
\end{enumerate}
\end{thm}

\begin{proof}
If $X \subseteq S$, then $\Omega(X,S,k) = \{0\}$ and $P(X,S,k) = G$,
so all the statements
are trivial in this case.  Hence we may assume that $X \not\subseteq S$.
We have that $G^0$ is $k$-defined and,
clearly, $Y_k(G)=Y_k(G^0)$ and $R_u(P_\lambda) = R_u(P_\lambda(G^0))$.
So we may assume that $G$ is connected.
We use Kempf's ``state formalism'', \cite[Sec.\ 2]{kempf}.
Actually we may consider states as only
defined on $k$-split subtori of $G$. First we need an analogue
over $k$ of \cite[Thm.~2.2]{kempf}. This is completely straightforward:
we simply work with $Y_k(G)$ instead of $Y(G)$ and use the conjugacy
of the maximal $k$-split tori of $G$ under $G(k)$, \cite[V.20.9(ii)]{Bo},
as in \cite{He}.  We also use the result that two $k$-defined parabolic subgroups of $G$ have a common maximal $k$-split torus \cite[V.20.7~Prop.]{Bo}.

Next we need a way to associate to a non-empty finite
subset $X_0\ne\{0\}$ of a rational $G$-module
$V_0$ a bounded admissible state. This is done as in \cite[2.4]{He}.
Then \cite[Lem.~3.2]{kempf} holds with $V$ and $v$ replaced by $V_0$ and $X_0$,
respectively.

Finally, we need to construct two bounded admissible states as in
\cite[Lem.~3.3]{kempf}. This is done precisely as in the proof of \emph{loc.\ cit.}
The embedding of $V$ in a rational $G$-module $V_0$ (denoted $V$ in \emph{loc.\ cit.}) and the
morphism $f:V\to W$,
$W$ a rational $G$-module, with the scheme-theoretic preimage $f^{-1}(0)$ equal to $S$, can be chosen as in \cite[Thm.~3.4]{kempf}.
Let $\Xi$ and $\Upsilon$ be the state of $X$ in $V_0$ and the state of
$f(X)$ in $W$, respectively.
Then assertions (i), (ii), (iii) and the first assertion of (iv)
follow as in \cite{kempf}.

The final assertion of (iv) is proved as follows.
Fix $\lambda\in\Omega(X,S,k)$.
Let $g \in N_{G(k)}(X)$.
Then $g\cdot \Omega(X,S,k)=\Omega(g\cdot X,g\cdot S,k)=\Omega(X,S,k)$,
by (iii) (note that $g\cdot S= S$).
So $g\cdot\lambda\in\Omega(X,S,k)$.
By the first assertion of (iv), $g\cdot\lambda=u\cdot\lambda$ for some
$u\in R_u(P(X,S,k))(k)$.
So $u^{-1}g\in C_G(\lambda(\ovl{k}^*)) = L_\lambda\subseteq P(X,S,k)$
and therefore
$g\in P(X,S,k)\cap G(k)=P(X,S,k)(k)$.
\end{proof}

\begin{defn}
\label{def:optimalparabolic}
We call $P(X,S,k)$ from Theorem \ref{thm:kempfrousseau} the
\emph{optimal destabilizing R-parabolic subgroup for $X$
with respect to $S$ over $k$}.  It is clear that $P(X,S,k)$ is a proper subgroup of $G$ if and only if $X\not\subseteq S$.  If $k$ is algebraically closed, then we often suppress the $k$ argument and write simply $\Omega(X,S)$ and $P(X,S)$.
\end{defn}

Next we discuss rationality properties of this construction.
If $X$ is uniformly $S$-unstable over $k$ and $k_1/k$
is a field extension, then $X$ is uniformly $S$-unstable
over $k_1$.  We want to investigate the relationship
between $P(X,S,k)$ and $P(X,S,k_1)$.  It appears
that one can say little in general if $k_1/k$ is
not separable, so we consider the special case when $k_1= k_s$.
We denote the $k$-closure of $X$ by $X^k$, cf.\ \cite[AG.11.3]{Bo}.
We obtain a rationality result as in \cite[Thm.\ 5.5]{He}.

We now choose a $G'$-invariant norm $\left\|\,\right\|$ on $Y(G)$. Note that this determines a $G(k_1)$-invariant norm on $Y_{k_1}(G)$ for any subfield $k_1$ of $\ovl k$.

\begin{thm}
\label{thm:kr-rationality}
Assume that $V$ is an affine $k$-variety
and that $S$ and the action of $G$ on $V$
are $k$-defined. Let $X$ be a non-empty subset of $V$.
Then the following hold:
\begin{enumerate}[{\rm(i)}]
\item $X$ is uniformly $S$-unstable over $k$
if and only if $X^k$ is uniformly $S$-unstable over $k_s$.
\item Assume that $X$ is uniformly $S$-unstable
over $k$ and that the norm $\left\|\,\right\|$ on $Y(G)$ is $k$-defined.
Then $\Omega(X,S,k)$ consists of the
$k$-defined cocharacters in $\Omega(X^k,S,k_s)$.
In particular, the cocharacters in $\Omega(X,S,k)$
are optimal for $X^k$ over $k_s$.
\end{enumerate}
\end{thm}

\begin{proof}
The embedding $V\hookrightarrow V_0$ and the morphism $f:V\to W$ of
the proof of Theorem~\ref{thm:kempfrousseau} can chosen to be
defined over $k$, see \cite[I.1.9]{Bo} and the proof of \cite[Lem.~1.1]{kempf}. One can then easily check
that for $\lambda\in Y_k(G)$ and any integer $r$,
\begin{equation}
\label{eqn:semicty}
 \mbox{the set $\{x\in V \mid \lambda\in\Lambda(x),a_{S,x}(\lambda)\ge r\}$ is
$k$-closed,}
\end{equation}
cf.\ the proof of \cite[Thm.\ 5.5]{He}.
It follows that $\Lambda(X,k)=\Lambda(X^k,k)$ and
that $a_{S,X}(\lambda)=a_{S,X^k}(\lambda)$ for all $\lambda\in\Lambda(X,k)$.

So we may assume that $X$ is $k$-closed.  We have to show that if $X$ is uniformly $S$-unstable over $k_s$, then $\Omega(X,S,k_s)$ contains a $k$-defined cocharacter.
If $Z$ is a $k$-variety (over $\ovl{k}$), then $\Gamma = \Gal(k_s/k)$
acts on the set
$Z$ and the $k$-closed subsets of $Z$ are the $\Gamma$-stable closed
subsets of $Z$; see \cite[11.2.8(ii)]{spr2}. Furthermore, if $Z_1$ and
$Z_2$ are $k$-varieties, then $\Gamma$ acts on the $k_s$-defined
morphisms from $Z_1$ to $Z_2$ and such a morphism is $k$-defined
if and only if it is fixed by $\Gamma$; see \cite[11.2.9]{spr2}.
So in our case $\Gamma$ acts on the sets $G$ and $V$ and $X$ is
$\Gamma$-stable. Now we can finish the proof as in
\cite[Thm.~4.2]{kempf} or \cite[Thm.~5.5]{He}.
\end{proof}

\begin{cor}
\label{cor:sepopt}
 Suppose the hypotheses of Theorem \ref{thm:kr-rationality} hold and that $\left\|\,\right\|$ is $k$-defined.  Let $k_1/k$ be a separable algebraic extension. Then $X$ is uniformly $S$-unstable over $k$ if and only if $X^k$ is uniformly $S$-unstable over $k_1$, and in this case we have $\Omega(X,S,k)= \Omega(X^k,S,k_1)\cap Y_k(G)$ and $P(X,S,k)= P(X^k,S,k_1)$.
\end{cor}

\begin{rems}
\label{rems:optimal}
(i). Hesselink's optimal class consists
in general of virtual cocharacters, since,
essentially, he requires $a_{S,X}(\lambda)=1$
(he minimizes the norm). We work with Kempf's
optimal class which consists of indivisible
cocharacters. There is an obvious bijection
between the two optimal classes.

(ii). If $k$ is not perfect,
then $X$ can be $S$-unstable
over $\ovl{k}$ but need not be $S$-unstable over $k$
(see Remark \ref{rem:obstacle}).  Even when $X$ is $S$-unstable over both $k$ and $\ovl{k}$, our methods do not tell us whether or not $P(X,S,k)= P(X,S)$ when $k$ is not perfect.

(iii). Assume that $k$ is perfect and that
$X=\{v\}$ with $v$ a $k$-point of $V$ outside $S$
whose $G$-orbit closure meets $S$.
Then Corollary \ref{cor:sepopt} gives
the existence of a $k$-defined destabilizing
cocharacter for $v$ and $S$ which is optimal over $\ovl{k}$.
This was first proved by Kempf in \cite[Thm.\ 4.2]{kempf}.
\end{rems}

Corollary \ref{cor:perfectorbitcrit} below and
Corollary \ref{cor:sepopt} answer
Question \ref{qn:cocharclsdext} for perfect $k$.

\begin{cor}
\label{cor:perfectorbitcrit}
Suppose that $k$ is perfect.
Let $V$ be an affine $G$-variety over $k$. Let $v\in V(k)$.
Then $G\cdot v$ is closed if and only if
$G(k)\cdot v$ is cocharacter-closed over $k$.
\end{cor}

\begin{proof}
If $G(k)\cdot v$ is not cocharacter-closed over $k$,
then $G\cdot v$ is not closed, by Corollary \ref{cor:perfectclsd}.
Conversely, suppose $G\cdot v$ is not closed.
Let $S$ be the unique closed $G$-orbit in $\ovl{G\cdot v}$.
Then $\ovl{G\cdot v}$ is $k$-defined (see, e.g.,\ \cite[1.9.1]{spr2}).
Let $\gamma\in\Gamma$. Then $\gamma(S)$ is a closed $G$-orbit which is
contained in $\ovl{G\cdot v}$, so it is equal to $S$. It follows that
$S$ is $\Gamma$-stable and therefore $k$-defined, since $k$ is perfect.
Now $v$ is $S$-unstable by the Hilbert-Mumford Theorem and therefore
$S$-unstable over $k$, by Theorem~\ref{thm:kr-rationality}(i).
Since $S\cap G\cdot v=\varnothing$, it is clear that
$G(k)\cdot v$ is not cocharacter-closed over $k$.
\end{proof}

\section{Applications to $G$-complete reducibility}
\label{sec:appl-gcr}
In this section we discuss some applications of the theory
developed in this paper,
with particular reference to Serre's concept of $G$-complete reducibility.
We briefly recall the definitions here; for more
details, see \cite{BMR}, \cite{serre2}.

\begin{defn}
\label{def:gcr}
A subgroup $H$ of $G$ is said to be \emph{$G$-completely reducible ($G$-cr)}
if whenever $H$ is contained in an R-parabolic subgroup $P$ of $G$,
there exists an R-Levi subgroup $L$ of $P$ containing $H$.
Similarly, a subgroup $H$ of $G$ is said to be
\emph{$G$-completely reducible over $k$}
if
whenever $H$ is contained in a $k$-defined
R-parabolic subgroup $P$ of $G$,
there exists a $k$-defined R-Levi subgroup $L$ of $P$ containing $H$.
\end{defn}

We have noted (Remark \ref{rem:krpars})
that not every $k$-defined R-parabolic subgroup of $G$
need stem from a cocharacter in $Y_k(G)$.
However, our next result shows that when considering questions of $G$-complete
reducibility over $k$, it suffices just to look at $k$-defined
R-parabolic subgroups of $G$ of the form $P_\lambda$
with $\lambda \in Y_k(G)$.

\begin{lem}
\label{lem:Gcroverk}
Let $H$ be a subgroup of $G$.
Then $H$ is $G$-completely reducible over $k$ if and only if
for every $\lambda \in Y_k(G)$ such that $H$ is contained in $P_\lambda$,
there exists $\mu \in Y_k(G)$ such that
$P_\lambda = P_\mu$ and $H \subseteq L_\mu$.
\end{lem}

\begin{proof}
Assume that for every $\lambda \in Y_k(G)$ such that $H$ is contained in $P_\lambda$,
there exists $\mu \in Y_k(G)$ such that $P_\lambda = P_\mu$ and $H \subseteq L_\mu$.
Let $\sigma\in Y(G)$ such that $P_\sigma$ is $k$-defined and $H\subseteq P_\sigma$.
After conjugating $\sigma$ by an element of $P_\sigma$, we may assume that
$\sigma\in Y(T)$ for some $k$-defined maximal torus $T$ of $P_\sigma$.
By Lemma~\ref{lem:Rparrat}(ii), there exists $\lambda\in Y_k(T)$ such that
$P_\sigma\subseteq P_\lambda$ and $P_\sigma^0=P_\lambda^0$.
Note that $L_\sigma=L_\lambda\cap P_\sigma$, by Lemma \ref{lem:Levidown}.
By assumption, there exists $\mu \in Y_k(G)$ such that $P_\lambda=P_\mu$ and $H \subseteq L_\mu$.  There exists $u\in R_u(P_\lambda)=R_u(P_\sigma)$ such that
$uL_\lambda u^{-1}=L_\mu$.
But then $L_{u\cdot\sigma}= uL_\sigma u^{-1}= u(L_\lambda\cap P_\sigma)u^{-1}=L_\mu\cap P_\sigma$ contains $H$. By Lemma~\ref{lem:Rparrat}(iii),
$L_{u\cdot\sigma}$ is $k$-defined, since
$L_{u\cdot\sigma}^0=L_\mu^0$ is $k$-defined.  Hence $H$ is $G$-cr over $k$.
The other implication follows from Corollary~\ref{cor:kLevi}.
\end{proof}

\begin{rem}
If $k$ is algebraically closed
(or even perfect, see \cite[Thm.~5.8]{BMR}) and $H$ is $k$-defined, then $H$
is $G$-cr over $k$ if and only if $H$ is $G$-cr.
\end{rem}

\subsection{Geometric criteria for $G$-complete reducibility}
\label{subsec:serre}
In \cite{BMR}, we show that $G$-complete reducibility has a geometric
interpretation in terms of the action of $G$ on $G^n$, the $n$-fold
Cartesian product of $G$ with itself, by simultaneous conjugation.
Let $\mathbf{h} \in G^n$
and let $H$ be the algebraic subgroup of $G$ generated by $\tuple{h}$.
Then $G\cdot\mathbf{h}$ is closed in $G^n$ if and only if
$H$ is $G$-cr \cite[Cor.\ 3.7]{BMR}.
To generalize this to subgroups that are not topologically finitely
generated, we need the following concept.

\begin{defn}
\label{def:generictuple}
Let $H$ be a subgroup of $G$ and let $G\hookrightarrow\GL_m$
be an embedding of algebraic groups.
Then $\tuple{h} \in H^n$ is called a
\emph{generic tuple of $H$ for the embedding $G\hookrightarrow\GL_m$}
if $\tuple{h}$ generates the associative subalgebra of $\Mat_m$
spanned by $H$. We call $\tuple{h}\in H^n$ a \emph{generic tuple of $H$}
if it is a generic tuple of $H$ for some embedding $G\hookrightarrow\GL_m$.
\end{defn}

Clearly, generic tuples exist for any embedding $G\hookrightarrow\GL_m$ if
$n$ is sufficiently large. The next lemma gives the main properties of
generic tuples.

\begin{lem}
\label{lem:generictuple}
Let $H$ be a subgroup of $G$,
let $\tuple{h}\in H^n$ be a generic tuple of $H$
for some embedding $G\hookrightarrow\GL_m$ and let
$H'$ be the algebraic subgroup of $G$ generated by $\tuple{h}$.
Then we have:
\begin{enumerate}[{\rm(i)}]
\item $C_M(\tuple{h})=C_M(H')=C_M(H)$ for any subgroup $M$ of $G$;
\item $H'$ is contained in the same R-parabolic
and the same R-Levi subgroups of $G$ as $H$;
\item If $H\subseteq P_\lambda$ for some $\lambda\in Y(G)$,
then $c_\lambda(\tuple{h})$ is a generic tuple of $c_\lambda(H)$
for the given embedding $G\hookrightarrow\GL_m$.
\end{enumerate}
\end{lem}

\begin{proof}
By assumption, $\tuple{h}$ generates the associative
subalgebra $A$ of $\Mat_m$ spanned by $H$.
For $\lambda\in Y(\GL_m)$ let $\mc P_\lambda$
be the subset of elements $x\in\Mat_m$ such
that $\underset{a \ra 0}{\lim}\, \lambda(a)\cdot x$ exists
and let $\mc L_\lambda$ be the centralizer of $\lambda(k^*)$ in $\Mat_m$.
Denote the limit morphism $\mc P_\lambda\to\mc L_\lambda$ by $c_\lambda$.
The well-known characterization of $\mc P_\lambda$ and $\mc L_\lambda$
in terms of flags of subspaces shows that they are subalgebras of $\Mat_m$ and that
$c_\lambda$ is a homomorphism of algebras. For $\lambda\in Y(G)$
we have $P_\lambda(G)=G\cap\mc P_\lambda$
and $L_\lambda(G)=G\cap\mc L_\lambda$.

(i).\ If a subset $S$ of $\Mat_m$ generates the associative subalgebra
$E$ of $\Mat_m$, then $C_M(S)=M\cap C_{\Mat_m}(E)$.
So $C_M(H)= C_M(H')=M\cap C_{\Mat_m}(A)=C_M(\tuple{h})$.

(ii).\ If a subset $S$ of $G$ generates the associative subalgebra
$E$ of $\Mat_m$, then $S\subseteq P_\lambda(G)$ if and only if
$E\subseteq{\mc P}_\lambda$, and $S\subseteq L_\lambda(G)$
if and only if $E\subseteq{\mc L}_\lambda$. This implies the assertion.

(iii).\ Since $c_\lambda:{\mc P}_\lambda\to{\mc L}_\lambda$ is a
homomorphism of associative algebras, $c_\lambda(\tuple{h})$
generates the associative subalgebra $c_\lambda(A)$ and this is also the
associative subalgebra of $\Mat_m$ generated by $c_\lambda(H)$.
\end{proof}

\begin{rem}
\label{rem:genericgenerators}
If $H$ is a subgroup of $G$ which is topologically
generated by a tuple ${\tuple h}\in H^n$,
then ${\tuple h}$
is a generic tuple of $H$ in the sense of Definition \ref{def:generictuple}.
To see this, consider an embedding $G\hookrightarrow\GL_m$.
Since the minimal polynomial of each $h_i$ has non-zero constant term,
we can express $h_i^{-1}$ as a polynomial in $h_i$.
Hence, if $A$ is the associative subalgebra of $\Mat_m$
generated by $\tuple h$, then $A$ contains the
inverses of each of the components $h_i$,
so it contains the subgroup of $\GL_m$
generated by $\tuple h$. But $A$ is closed, so it contains $H$.
\end{rem}

\begin{rem}
\label{rem:generictuple}
 Let $H$ be a $k$-defined subgroup of $G$.  Even if $H$ is topologically finitely generated, there need not exist a $k$-defined generating tuple.  The notion of a generic tuple lets us get around this problem.  Note that if $\tuple{h}$ is a generic tuple of $H$, then $C_{G(k_s)}(\tuple{h})$ is $\Gamma$-stable by Lemma \ref{lem:generictuple}(i), which is a sufficient condition for many of the results in Section~\ref{sec:orbrat} to hold.  Another advantage of generic tuples is that one can extend the action of $S_n$ on an $n$-tuple by permutation of the components (cf.\ \cite[Thm.~5.8]{BMR}) to an action of ${\rm GL}_n(k)$ (cf. \cite[Sec.~3.8]{BMRT:relative}).
\end{rem}

The connection between $G$-complete reducibility and $G$-orbits of tuples
is made transparent by part (iii) of
the following theorem which is, essentially,
a consequence of Theorem~\ref{thm:Ruconj}. It also shows how
statements about generic tuples can be translated back into
statements about subgroups of $G$.
Note that, in view of Remark \ref{rem:genericgenerators},
the final statement of Theorem \ref{thm:orbclosconjcrit}(iii)
recovers \cite[Cor.\ 3.7]{BMR}.

\begin{thm}\
\label{thm:orbclosconjcrit}
\begin{enumerate}[{\rm (i)}]
\item Let $n\in {\mathbb N}$, let $\tuple{h}\in G^n$ and
let $\lambda\in Y(G)$ such that
$\tuple{m}:=\lim_{a\ra 0}\lambda(a)\cdot \tuple{h}$ exists.
Then the following are equivalent:
\begin{enumerate}[{\rm (a)}]
\item $\tuple{m}$ is $G$-conjugate
to $\tuple{h}$;
\item $\tuple{m}$ is $R_u(P_\lambda)$-conjugate to $\tuple{h}$;
\item $\dim G\cdot\tuple{m}=\dim G\cdot\tuple{h}$.
\end{enumerate}
\item Let $H$ be a subgroup of $G$ and let $\lambda\in Y(G)$.
Suppose $H\subseteq P_\lambda$ and set $M= c_\lambda(H)$.
Then $\dim C_G(M)\geq \dim C_G(H)$ and the following are equivalent:
\begin{enumerate}[{\rm (a)}]
\item $M$ is $G$-conjugate to $H$;
\item  $M$ is $R_u(P_\lambda)$-conjugate to $H$;
\item  $H$ is contained in an R-Levi subgroup of $P_\lambda$;
\item $\dim C_G(M)=\dim C_G(H)$.
\end{enumerate}
\item Let $H$, $\lambda$ and $M$ be as in (ii) and let
$\tuple{h}\in H^n$ be a generic tuple of $H$. Then the
assertions in (i) are equivalent to those in (ii).
In particular, $H$ is $G$-completely reducible
if and only if $G\cdot\tuple{h}$ is closed in $G^n$.
\end{enumerate}
\end{thm}

\begin{proof}
(i). It is obvious that (b) implies (a) and (a) implies (c).  It follows immediately from
Theorem~\ref{thm:Ruconj} and \cite[Prop.~I.1.8]{Bo} that (c) implies (b).

(ii) and (iii). Let $\tuple{h}\in H^n$, let $H'$ be the algebraic subgroup of
$G$ generated by $\tuple{h}$ and let $\lambda\in Y(G)$.
Then $\lim_{a\ra 0}\lambda(a)\cdot \tuple{h}$ exists
if and only if $H'\subseteq P_\lambda$.
Now assume that $\tuple{m}=\lim_{a\ra 0}\lambda(a)\cdot \tuple{h}$ exists.
Let $u\in R_u(P_\lambda)$. Then $\tuple{h}=u\cdot\tuple{m}$
if and only if $u\cdot\lambda$ fixes $\tuple{h}$ (Lemma~\ref{lem:Ruconj})
if and only if $H'\subseteq L_{u\cdot\lambda}=uL_\lambda u^{-1}$.
Pick a generic tuple $\tuple{h}\in H^n$ of $H$ for some $n\in {\mathbb N}$.
Then $\tuple{m}=c_\lambda(\tuple{h})$ is a generic tuple of $M$,
by Lemma~\ref{lem:generictuple}(iii).
Now the first assertion of (ii) follows from the fact that
$\dim G\cdot\tuple{m}\le\dim G\cdot\tuple{h}$
(see \cite[Prop.~I.1.8]{Bo}),
since $\dim G\cdot\tuple{h}=\dim G-\dim C_G(\tuple{h})$, which equals $\dim G-\dim C_G(H)$ (Lemma \ref{lem:generictuple}(i)),
and likewise for $\tuple{m}$. Now we prove the equivalences.
Clearly, (b) implies (a) and (a) implies (d).
Furthermore, we have for $u\in R_u(P_\lambda)$
that $H\subseteq L_{u\cdot\lambda}$ if and only if
$H=c_{u\cdot\lambda}(H)=uMu^{-1}$.
So (b) is equivalent to (c). Now assume that (d) holds.
Then $\dim G\cdot\tuple{m}=\dim G\cdot\tuple{h}$.
So $\tuple{m}$ is $R_u(P_\lambda)$-conjugate to $\tuple{h}$, by (i), whence $H'$ is $R_u(P_\lambda)$-conjugate to $c_\lambda(H')$.
By the equivalence of (b) and (c) (applied to $H'$), $H'$ is contained
in an R-Levi subgroup of $P_\lambda$. Since $\tuple{h}$ is a generic tuple of $H$,
(c) holds by Lemma \ref{lem:generictuple}(ii).
Lemma \ref{lem:generictuple}(i) implies that (i)(c) and (ii)(d) are equivalent, so the first assertion of (iii) holds.
The final assertion of (iii) follows from the first, Lemma \ref{lem:generictuple}(ii) and the
Hilbert-Mumford Theorem.
\end{proof}

We now give a geometric characterization of $G$-complete reducibility over an arbitrary field $k$, using Theorem \ref{thm:cocharclosedcrit}.  Note that the subgroup $H$ in Theorem \ref{thm:cocharclosedcritforGcr} need not be $k$-defined.  In view of Remark \ref{rem:cocharclosure},
Theorem \ref{thm:cocharclosedcritforGcr} in the special case $k = \ovl{k}$
yields the final assertion of Theorem \ref{thm:orbclosconjcrit}(iii).

\begin{thm}
\label{thm:cocharclosedcritforGcr}
Suppose that $G$ is connected.
Let $H$ be a subgroup of $G$ and let $\tuple h \in H^n$
be a generic tuple of $H$.
Then
$H$ is $G$-completely reducible over $k$
if and only if
$G(k)\cdot \tuple h$ is cocharacter-closed over $k$.
\end{thm}

\begin{proof}
Suppose that $G(k)\cdot \tuple h$ is cocharacter-closed over $k$.
In order to show that $H$ is $G$-cr over $k$,
we just need to consider R-parabolic subgroups of $G$ containing $H$ of the
form $P_\lambda$ with $\lambda \in Y_k(G)$, by Lemma \ref{lem:Gcroverk}.
Let $\lambda \in Y_k(G)$ be such that $P_\lambda$
contains $H$.
Then $\tuple h': = c_\lambda(\tuple h)$ exists.
Since $G(k)\cdot \tuple h$ is cocharacter-closed over $k$,
there exists $u \in R_u(P_\lambda)(k)$ such that
$\tuple h' = u \cdot \tuple h$, by Theorem \ref{thm:cocharclosedcrit}.
By Lemma \ref{lem:Ruconj},
$u^{-1}\cdot\lambda$ centralizes $\tuple h$.
Hence $H \subseteq L_{u^{-1}\cdot\lambda}$.
Since $L_{u^{-1}\cdot\lambda}$ is $k$-defined,
$H$ is $G$-completely reducible over $k$.

Now assume that $H$ is $G$-completely reducible over $k$.
Let $\lambda \in Y_k(G)$
such that $\tuple h': = c_\lambda(\tuple h)$ exists.
Then $H \subseteq P_\lambda$.
So, by hypothesis, there exists a $k$-defined R-Levi subgroup $L$ of
$P_\lambda$ with $H \subseteq L$.
By Lemma \ref{lem:Rparrat}(iii),
there exists $u \in R_u(P_\lambda)(k)$ such that
$L = u^\inverse L_\lambda u = L_{u^\inverse \cdot \lambda}$.
Hence $u^\inverse \cdot \lambda$ centralizes $H$ and so
$u^\inverse \cdot \lambda$ centralizes $\tuple h$.
Thus, by Lemma \ref{lem:Ruconj}, we have $\tuple h' = u \cdot \tuple h$.
Consequently,
$G(k)\cdot \tuple h$ is cocharacter-closed over $k$.
\end{proof}

\begin{rem}
\label{rem:obstacle}
We now provide examples for the failure of Question \ref{qn:cocharclsdext}
in general.
In \cite[Ex.~7.22]{BMRT},
we give an example of a reductive group $G$ and a subgroup $H$,
both $k$-defined, such that $H$ is $G$-completely
reducible but not $G$-completely reducible over $k$.
Let ${\mathbf h}\in H^n$ be a generic tuple of $H$.
Then, by Theorem \ref{thm:cocharclosedcritforGcr},
$G\cdot {\mathbf h}$ is closed in $G^n$ but $G(k)\cdot {\mathbf h}$ is not
cocharacter-closed over $k$.
Conversely, an example due to McNinch, \cite[Ex.~5.11]{BMR}, gives
a reductive group $G$ and a subgroup $H$, both $k$-defined,
such that $H$ is $G$-completely reducible over $k$ but not $G$-completely
reducible, and this implies that there exists a generic tuple
${\mathbf h}\in H^n$ for some
$n\in {\mathbb N}$ such that $G(k)\cdot {\mathbf h}$ is cocharacter-closed
over $k$ but $G\cdot {\mathbf h}$ is not closed.  Hence $\tuple{h}$ is uniformly $S$-unstable over $\ovl{k}$ but not uniformly $S$-unstable over $k$, where $S$ is the unique closed $G$-orbit contained in $\ovl{G\cdot v}$.  In fact, we have $S= \{(1,\ldots, 1)\}$ in this example, so $S$ has a $k$-point.
Note that $C_{G(k_s)}(\tuple{h})$ is $\Gamma$-stable in both cases (this follows from Lemma \ref{lem:generictuple}(i)), so we have counterexamples to Question \ref{qn:cocharclsdext}; moreover, in both cases the extension $\ovl{k}/k$ is not separable.

We even have an example where $v\in V(k)$, $k$ is infinite, $G\cdot v$ is not closed and $G(k)\cdot v$ is a Zariski-closed subset of $V(k)$.  Let $k$ be a separably closed non-perfect field of characteristic 2 and let $G= {\rm GL}_2$ acting on $V= {\rm GL}_2$ by conjugation.  Choose $a\in k^{1/2}\setminus k$.  Let $v= \twobytwo{0}{1}{a^2}{0}$ and let $v'= \twobytwo{a}{0}{0}{a}$ (cf. \cite[2.4.11]{spr2}).  It is easily checked that the closure of $G\cdot v$ is $G\cdot v\cup \{v'\}$.  Moreover, the orbit map $G\ra G\cdot v$, $g\mapsto g\cdot v$ is separable (cf.\ \cite[Ex.~3.28 and Rem.~3.31]{BMR}) and hence is surjective on $k$-points \cite[AG.13.2~Thm.]{Bo}.  This implies that $G(k)\cdot v$ is closed in $V(k)$ (and hence is cocharacter-closed over $k$, by Remark \ref{rem:cocharclosure}(ii)).  The unique closed $G$-orbit $S$ contained in $\ovl{G\cdot v}$ has no $k$-points --- in contrast to the previous example --- and it follows that $v$ is uniformly $S$-unstable over $\ovl{k}$ but not uniformly $S$-unstable over $k$.
\end{rem}

The interpretation of $G$-complete reducibility in terms of
orbits allows us to provide a partial answer to a question of Serre;
for a more general result, see \cite[Thm.~4.13]{BMRT:relative}.
Let $k_1/k$ be a separable algebraic extension of fields.
Serre has asked whether it is the case that a
$k$-defined subgroup $H$ of $G$ is $G$-completely reducible over $k$
if and only if it is $G$-completely reducible over $k_1$.  This was proved in \cite[Thm.~5.8]{BMR} for $k$ perfect by passing back and forth between $k$ and $\ovl{k}$ and between $k_1$ and $\ovl{k}$.  In general this approach fails because the extension $\ovl{k}/k$ need not be separable; we discuss this further in Example \ref{exmp:obstacle} below.  This shows that even if one is interested only in separable field extensions $k_1/k$, problems with inseparability can arise.

We can now answer one direction of Serre's question.  Theorem \ref{thm:serresquestion}
gives a group-theoretic analogue of
Theorem \ref{thm:GeneralSeparableDown}.

\begin{thm}
\label{thm:serresquestion}
Suppose $k_1/k$ is a separable extension of fields.
Let $H$ be a $k$-defined subgroup of $G$.
If $H$ is $G$-completely reducible over $k_1$, then
$H$ is $G$-completely reducible over $k$.
\end{thm}

\begin{proof}
Let $\mathbf{h} \in H^n$  be a generic tuple of $H$ for some $n$.
Suppose $\lambda \in Y_k(G)$ is such that $H \subseteq P_\lambda$.
Then since $H$ is $G$-cr over $k_1$,
there exists
$u_1 \in R_u(P_\lambda)(k_1) \subseteq R_u(P_\lambda)(k_s)$ such that
$H \subseteq L_{u_1\inverse\cdot\lambda}$.
Thus, $u_1\inverse\cdot\lambda$ centralizes $H$ and so
$u_1\inverse\cdot\lambda$ centralizes $\tuple h$. It thus follows from
Lemma \ref{lem:Ruconj} that
$\lim_{a\to 0} \lambda(a)\cdot \mathbf{h} = u_1\cdot \mathbf{h}$.
Now $C_{G(k_s)}(\mathbf{h})$ is $\Gamma$-stable by Remark \ref{rem:generictuple}, so we can apply Theorem \ref{thm:sepRuconj} to conclude that there exists
$u \in R_u(P_\lambda)(k)$ such that
$\lim_{a\to 0} \lambda(a)\cdot \mathbf{h} = u\cdot \mathbf{h}$.
Thus $u\inverse\cdot\lambda \in Y_k(G)$ centralizes $\mathbf{h}$ (Lemma \ref{lem:Ruconj}), whence $u\inverse\cdot\lambda$ centralizes $H$ (Lemma \ref{lem:generictuple}(ii)).
We therefore have $H \subseteq L_{u\inverse\cdot \lambda}$, a $k$-defined
R-Levi subgroup of $P_\lambda$, as required.
\end{proof}

\begin{exmp}
\label{ex:serreGL}
We show that the answer to Serre's question is yes when $G = \GL(V)$,
where $V = \ovl{k}^n$ with the standard $k$-structure $k^n$ on $V$.
This of course determines the usual $k$-structure on $\GL(V)$.
Let $H$ be a subgroup of $G$ and let $A$ be its enveloping algebra: that is,
the $\ovl{k}$-span of $H$ in $\End_{\ovl{k}}(V)$.
Then $A$ is $k$-defined provided $H$ is. To see this, we exhibit
a $\Gamma$-stable, dense subset of separable points in $A$
and for this set we simply take the $k_s$-span of $H(k_s)$ in
$\End_{\ovl{k}}(V)$.
As a consequence, we obtain the following characterization of
$\GL(V)$-complete reducibility over $k$ under the assumption that $H$ is $k$-defined:
$H$ is $\GL(V)$-cr over $k$ if and only if $V(k)=k^n$
is a semisimple $A(k)$-module (if and only if $A(k)$ is a
semisimple algebra). Here $A(k)$ denotes the algebra of $k$-points of $A$
(this is a $k$-structure on $A$: $A = \ovl{k}\otimes_k A(k)$).

Finally, if $H$ and $A$ are as above and $k_1\subseteq \ovl{k}$ is an algebraic
extension of $k$, then $A(k_1) = k_1\otimes_k A(k)$.
It follows from \cite[Cor.\ 69.8 and Cor.\ 69.10]{curtisreiner}
that $A(k)$ is semisimple if and only if $A(k_1)$
is semisimple,  provided $k_1$ is a separable
extension of $k$. By the above this means that $H$ is $\GL(V)$-cr over $k$
if and only if $H$ is $\GL(V)$-cr over $k_1$.
\end{exmp}

\subsection{Optimal destabilizing parabolic subgroups for subgroups of $G$}
\label{subsec:optnongcr}

In this section we assume $G$ is a normal $k$-subgroup
of a 
$k$-defined linear algebraic group $G'$.  If $G'$ is not explicitly given, we just take $G'$ to be $G$.
We fix a $G'$-invariant norm $\left\|\,\right\|$ on $Y(G)$,
see Definition \ref{def:norm}.  Recall our convention that $P_\lambda$ is a subgroup of $G$: so the optimal destabilizing subgroups defined below are parabolic subgroups of $G$, not of $G'$.

Let $H$ be a subgroup of $G$ such that $H$ is not $G$-completely reducible.
Suppose there exists $\tuple{h}\in H^n$ such
that $H$ is generated by $\tuple{h}$.  Then $G\cdot \tuple{h}$ is not closed in $G^n$, and we can construct the optimal destabilizing
parabolic subgroup $P_\tuple{h}= P(\tuple{h},S)$ of $G$ for $\tuple{h}$, where $S$ is the unique closed $G$-orbit contained in $\ovl{G\cdot \tuple{h}}$.
Several recent results involving $G$-complete reducibility
have rested on this construction \cite{martin2},
\cite[Sec.~3, Thm.~5.8]{BMR}, \cite[Thm.~5.4(a)]{BMRT}.
There are some technical problems in applying it.
For instance,
if $g\in G$ normalizes $H$, then $g$ need not centralize $\tuple{h}$ (cf.\ the proof of \cite[Prop.~5.7]{BMRT}).

We now show how to associate an optimal destabilizing
R-parabolic subgroup $P(H)$ to $H$ using uniform $S$-instability.
This avoids the above problems and yields shorter,
cleaner proofs, because we need not deal explicitly
with a generating tuple for $H$.

\begin{rem}
\label{rem:BorelTits}
We can regard the following construction as a generalization of
the Borel-Tits construction \cite{boreltits}, which associates to a
non-trivial unipotent element $u\in G$ a parabolic subgroup
$P_{\rm BT}$ of $G$ such that $u\in R_u(P_{\rm BT})$.
More generally, the latter construction associates to a
non-reductive subgroup $H$ of $G$ a parabolic
subgroup $P_{\rm BT}$ of $G$ such that $R_u(H)\subseteq R_u(P_{\rm BT})$.
Our construction works for any non-G-completely reducible $H$,
including the case when $H$ is reductive. Note, however, that
if $H$ is non-reductive, then $P_{\rm BT}$ does not necessarily
coincide with $P(H)$ from Definition \ref{defn:optpar}, \cite[Rem.\ 8.4]{He}.
\end{rem}

First we need a prelimimary result which gives us a closed $G$-stable subvariety $S_n(M)$ of $G^n$ to work with.  The idea is that the $G$-conjugacy class of a generic tuple of the group $M$ from Proposition \ref{prop:uniquegcr}
corresponds to the unique closed $G$-orbit in the $G$-orbit closure of a generic tuple of $H$.

\begin{prop}
\label{prop:uniquegcr}
Let $H$ be a subgroup of $G$.
\begin{enumerate}[{\rm(i)}]
\item There exists $\lambda\in Y(G)$ and a
$G$-completely reducible subgroup $M$ of $G$
such that $H\subseteq P_\lambda$ and $c_\lambda(H)=M$.
Moreover, $M$ is unique up to $G$-conjugacy and its
$G$-conjugacy class depends only on the $G$-conjugacy class of $H$.
\item Any automorphism of the algebraic group $G$
that stabilizes the $G$-conjugacy class of $H$,
stabilizes the $G$-conjugacy class of $M$.
\item Any $\gamma\in \Gamma$ that stabilizes the $G$-conjugacy class of $H$,
stabilizes the $G$-conjugacy class of $M$.
\item If $\mu\in Y(G)$ and $H\subseteq P_\mu$, then the
procedure described in (i) associates the same
$G$-conjugacy class of subgroups to $H$ and $c_\mu(H)$.
\end{enumerate}
\end{prop}

\begin{proof}
(i).\ Let $P_\lambda$ be an R-parabolic subgroup
of $G$ which is minimal with respect to containing $H$.
Since $H\subseteq c_\lambda(H)R_u(P_\lambda)$ and
$R_u(P_\lambda)\subseteq  R_u(Q)$
for every R-parabolic subgroup $Q$ of $G$ with $Q\subseteq P_\lambda$,
we have that $P_\lambda$ is also minimal with respect to containing
$M=c_\lambda(H)$. So $M$ is $L_\lambda$-irreducible and
therefore $G$-cr, see \cite[Cor.~6.4, Cor.~3.22]{BMR}.

Now suppose $\lambda,\mu\in Y(G)$ such that
$H\subseteq P_\lambda$ and $H\subseteq P_\mu$, and such that
$M_1=c_\lambda(H)$ and $M_2=c_\mu(H)$ are $G$-cr.
Since $P_\lambda$ and $P_\mu$ have a maximal torus in common
(see, e.g., \cite[Cor.~IV.14.13]{Bo}),
after possibly replacing $M_1$ by an $R_u(P_\lambda)$-conjugate and
$M_2$ by an $R_u(P_\mu)$-conjugate, we may assume that $\lambda(k^*)$
and $\mu(k^*)$ commute. Clearly, $P_\lambda\cap P_\mu$ is stable under
$c_\lambda$ and $c_\mu$. It follows from \cite[Lem.~6.2(iii)]{BMR} that,
on $P_\lambda\cap P_\mu$, the composition
$c_\lambda\circ c_\mu=c_\mu\circ c_\lambda$ is
the projection $P_\lambda\cap P_\mu \to L_\lambda\cap L_\mu$ with kernel
$R_u(P_\lambda\cap P_\mu)$. So $c_\lambda(M_2)=c_\mu(M_1)$.
Now $M_1$ is $G$-cr, so, by
Theorem~\ref{thm:orbclosconjcrit}(ii),
$M_1$ is $R_u(P_\mu)$-conjugate to $c_\mu(M_1)$.
Similarly, $M_2$ is $R_u(P_\lambda)$-conjugate to $c_\lambda(M_2)$.
So $M_1$ and $M_2$ are $G$-conjugate. Finally, we observe that if
$H\subseteq P_\lambda$ and $g\in G$, then $gHg^{-1}\subseteq P_{g\cdot\lambda}$
and $c_{g\cdot\lambda}(gHg^{-1})=gc_\lambda(H)g^{-1}$,
so the $G$-conjugacy class of $M$ only depends on that of $H$.

(ii).\ Let $\varphi$ be an automorphism of the algebraic
group $G$ that stabilizes the $G$-conjugacy class of $H$
and let $\lambda\in Y(G)$ such that $H\subseteq P_\lambda$
and $c_\lambda(H)$ is $G$-cr. Then $\varphi(H)$ is $G$-conjugate
to $H$ and $\varphi(H)\subseteq P_{\varphi\circ\lambda}$.
Now $\varphi(c_\lambda(H))=c_{\varphi\circ\lambda}(\varphi(H))$
is $G$-conjugate to $c_\lambda(H)$ by (i), since $c_\lambda(H)$ is $G$-cr.

(iii).\ Let $\gamma\in \Gamma$ such that $\gamma$ stabilizes the $G$-conjugacy class of $H$ and let $\lambda\in Y(G)$ such that $H\subseteq P_\lambda$
and $c_\lambda(H)$ is $G$-cr. Then $\gamma\cdot H$ is $G$-conjugate
to $H$ and $\gamma\cdot H\subseteq P_{\gamma\cdot \lambda}$.
Now $\gamma\cdot M$ is $G$-cr by Lemma \ref{lem:Galoislim}, so $\gamma\cdot M= \gamma\cdot(c_\lambda(H))=c_{\gamma\cdot \lambda}(\gamma\cdot H)$
is $G$-conjugate to $c_\lambda(H)$ by (i).

(iv).\ Assume that $H\subseteq P_\mu$ and let $\lambda\in Y(G)$
such that $H\subseteq P_\lambda$ and $c_\lambda(H)$ is $G$-cr.
After replacing $\lambda$ by a $P_\lambda$-conjugate and $\mu$
by a $P_\mu$-conjugate, we may assume that $\lambda$ and $\mu$
commute. As in (i), $P_\lambda\cap P_\mu$ is stable under $c_\lambda$
and $c_\mu$ and $c_\lambda(c_\mu(H))=c_\mu(c_\lambda(H))$ is
$R_u(P_\mu)$-conjugate to $c_\lambda(H)$ and is $G$-cr,
since $c_\lambda(H)$ is $G$-cr.
\end{proof}

\begin{defn}
\label{def:SnM}
Let $M$ be a subgroup of $G$. Given $n\in {\mathbb N}$,
set $S_n(M) := \ovl{G\cdot M^n}$, a closed $G$-stable subset
of $G^n$. Note that $S_n(M)$ only depends on the $G$-conjugacy class of $M$.
Now suppose there exists $\lambda\in Y_k(G)$ such that
$H\subseteq P_\lambda$ and $M=c_\lambda(H)$. Note that
if $k$ is algebraically closed, then some subgroup of $G$ in the
$G$-conjugacy class attached to $H$ of $G$-cr subgroups of $G$,
provided by Proposition~\ref{prop:uniquegcr}, satisfies
this hypothesis. Then we have $c_\lambda(H^n)\subseteq M^n\subseteq S_n(M)$,
so $H^n$ is uniformly $S_n(M)$-unstable over $k$
(in the sense of Definition \ref{def:destabilizing}).
\end{defn}

\begin{thm}
\label{thm:optpar}
Let $G$, $G'$ and $\left\|\,\right\|$ be as above.
Let $H$ be any subgroup of $G$ and let $n\in {\mathbb N}$
such that $H^n$ contains a generic tuple of $H$.
Let $M$ be a subgroup of $G$ and suppose that
$M=c_\lambda(H)$ for some $\lambda\in Y_k(G)$
with $H\subseteq P_\lambda$. Put $\Omega(H,M,k) := \Omega(H^n,S_n(M),k)$.
Then the following hold:
\begin{enumerate}[{\rm (i)}]
\item $P_\mu = P_\nu$ for all $\mu, \nu \in \Omega(H,M,k)$.
Let $P(H,M,k)$ denote the unique R-parabolic subgroup of $G$ so defined.
Then $H\subseteq P(H,M,k)$ and $R_u(P(H,M,k))(k)$
acts simply transitively on $\Omega(H,M,k)$.
\item For $g\in G'(k)$ we have
$\Omega(gHg^{-1},gMg^{-1},k)=  g\cdot\Omega(H,M,k)$
and $P(gHg^{-1},gMg^{-1},k)=  gP(H,M,k)g^{-1}$.
If $g\in G(k)$ normalizes $H$,
then $g\in P(H,M,k)$.
\item If $\mu\in\Omega(H,M,k)$, then $\dim C_G(c_\mu(H))\ge\dim C_G(M)$.
If $M$ is $G$-conjugate to $H$, then $\Omega(H,M,k)=\{0\}$ and $P(H,M,k)= G$. If $M$
is not $G$-conjugate to $H$, then $H$ is not contained
in any R-Levi subgroup of $P(H,M,k)$.
\end{enumerate}
\end{thm}

\begin{proof}
(i) and (ii). Clearly, $H^n$ is uniformly $S_n(M)$-unstable over $k$, so $\Omega(H,M,k)$ is well-defined.  If $\mu\in \Omega(H,M,k)$, then $\lim_{a\ra 0} \mu(a)\cdot \tuple{h}$ exists for all $\tuple{h}\in H^n$, so $H\subseteq P_\mu= P(H,M,k)$.  The rest follows immediately from Theorem \ref{thm:kempfrousseau}.

(iii).
We have $\dim C_G(\tuple{m})\ge\dim C_G(M)$ for all $\tuple{m}\in G\cdot M^n$.
Since $\tuple{m}\mapsto\dim C_G(\tuple{m})$ is upper semi-continuous,
cf.~\cite[Lem.~3.7(c)]{newstead}, this inequality holds
for all $\tuple{m}\in S_n(M)$.

Let $\mu\in\Omega(H,M,k)$. Let $\tuple{h}\in H^n$ be a generic tuple of $H$.
Then $c_\mu(\tuple{h})$ is a generic tuple of $c_\mu(H)$,
by Lemma~\ref{lem:generictuple}(iii).
So $\dim C_G(c_\mu(H))=\dim C_G(c_\mu(\tuple{h}))\ge\dim C_G(M)$,
since $c_\mu(\tuple{h})\in S_n(M)$.

It follows easily from the definitions that $P(H,M,k)=G$
if and only if $\Omega(H,M,k)=\{0\}$ if and only if $H^n\subseteq S_n(M)$.
Clearly, the latter is the case if $M$ is $G$-conjugate to $H$.
Now assume that $M$ is not $G$-conjugate to $H$ and pick
$\mu\in\Omega(H,M,k)$. Then $\dim C_G(M)>\dim C_G(H)$,
by Theorem~\ref{thm:orbclosconjcrit}(ii) (applied to $\lambda$).
So $\dim C_G(c_\mu(H))>\dim C_G(H)$, by the above and $H$ is not
contained in any R-Levi subgroup of $P(H,M,k)$,
by Theorem~\ref{thm:orbclosconjcrit}(ii) (applied to $\mu$).
\end{proof}

\begin{defn}
\label{defn:optpar}
We call $\Omega(H,M,k)$ the
\emph{optimal class for $H$ with respect to $M$ over $k$}
and we call $P(H,M,k)$ the
\emph{optimal destabilizing R-parabolic subgroup for $H$ with respect to $M$ over $k$}.
Assume the $G$-conjugacy class given by Proposition~\ref{prop:uniquegcr}
contains a group $M$ of the form $c_\lambda(H)$
for some $\lambda\in Y_k(G)$.
Then we set $\Omega(H,k):= \Omega(H,M,k)$ and $P(H,k):= P(H,M,k)$.
Under this assumption we have, by Proposition~\ref{prop:uniquegcr}
and Theorem~\ref{thm:optpar}, that $N_{G(k)}(H)$ is contained in $P(H,k)$
and that for $\mu\in\Omega(H,k)$, $c_\mu(H)$ is $G$-completely reducible.
So, by Theorem~\ref{thm:orbclosconjcrit}(ii), if $H$ is not $G$-completely reducible, then it is not contained in any R-Levi subgroup of $P(H,k)$. Note that, trivially, $P(H,k)=G$ if $H$ is $G$-completely reducible.
We call $\Omega(H,k)$ the
\emph{optimal class for $H$ over $k$}
and we call $P(H,k)$ the
\emph{optimal destabilizing R-parabolic subgroup for $H$ over $k$}.  We suppress the dependence on the choice of $n$ and $\left\|\,\right\|$ in the notation (cf.\ Remark \ref{rem:independence}).

Note that the assumption of the previous paragraph is satisfied if $k$ is algebraically closed.
In that case we usually suppress
the $k$ argument and write simply $\Omega(H)$ and $P(H)$ instead; we refer to these as the
\emph{optimal class for $H$}
and the
\emph{optimal destabilizing R-parabolic subgroup for $H$},
respectively.
\end{defn}

We now suppose that the fixed
norm $\left\|\,\right\|$ on $Y(G)$
is $k$-defined, cf.\ Definition \ref{def:norm}.
We get the following rationality result.

\begin{thm}
\label{thm:potrat}
Let $G$, $G'$, $H$ and $n$ be as in Theorem~\ref{thm:optpar}
and assume that $H$ is $k$-closed.
Then the following hold:
\begin{enumerate}[{\rm (i)}]
\item Suppose that $M$ is a subgroup of $G$ such that
$M=c_\lambda(H)$ for some $\lambda\in Y_{k_s}(G)$ with
$H\subseteq P_\lambda$ and such that $S_n(M)$ is $k$-defined
(this is the case in particular if $M$ is $k$-defined).
Then $\Omega(H,M,k)$ is well-defined and equal to $\Omega(H,M,k_s)\cap Y_k(G)$.  Moreover, $P(H,M,k)$ is well-defined and equal to $P(H,M,k_s)$.  In particular, $P(H,M,k_s)$ is $k$-defined.
\item If $k$ is perfect, then $\Omega(H,k)$ is well-defined and equal to $\Omega(H)\cap Y_k(G)$.  Moreover, $P(H,k)$ is well-defined and equal to $P(H)$.  In particular, $P(H)$ is $k$-defined.
\end{enumerate}
\end{thm}

\begin{proof}
(i).\ This follows immediately from Theorem \ref{thm:kr-rationality}.

(ii).\ Since $k$ is perfect, $k_s=\ovl k$.  Let $M$ be as in Proposition \ref{prop:uniquegcr}; then $H$ is uniformly $S_n(M)$-unstable over $\ovl{k}$.  Now $S_n(M)$ is $\Gamma$-stable by Proposition \ref{prop:uniquegcr}(iii) and hence is $k$-defined, since $k$ is perfect.  The result now follows from (i).
\end{proof}

\begin{rems}
\label{rems:optpar}
(i). Let $M$ be as in Theorem~\ref{thm:optpar}
and let $M_0$ be a $G$-cr subgroup from the
$G$-conjugacy class associated to $H$ by
Proposition~\ref{prop:uniquegcr}. Then we
have $S_n(M_0)\subseteq S_n(M)$ for any $n$.
To prove this we may, by the final assertion
in Proposition~\ref{prop:uniquegcr}, assume
that $M=H$. Furthermore, we may assume that
$M_0=c_\lambda(H)$ for some $\lambda\in Y(G)$
with $H\subseteq P_\lambda$.
Since $\lambda(a)\cdot H^n\subseteq G\cdot H^n$
for all $a\in k^*$, we have $M_0^n= c_\lambda(H^n)\subseteq S_n(H)$.
So $S_n(M_0)\subseteq S_n(H)$.

(ii). Note that $G\cdot M^n$ need not be closed:
e.g., take $G$ connected and non-abelian, $n$ to be 1 and $H = M$
to be a maximal torus of $G$.
\end{rems}

\begin{exmp}
We give an example of the usefulness of this construction
(cf.\ \cite{martin2} and \cite[Thm.~3.10]{BMR}).
Let $H$ be a $G$-completely reducible subgroup of
$G$ and let $N$ be a normal subgroup of $H$.
We prove that $N$ is $G$-completely reducible.
Suppose not.  Then $H\subseteq N_G(N)\subseteq P(N)$.
Since $N$ is not contained in an R-Levi subgroup of $P(N)$,
neither is $H$.  But $H$ is assumed to be $G$-completely reducible,
so this is impossible.  We deduce that $N$ is $G$-completely reducible
after all.
\end{exmp}
Here is a second example, which illustrates
the gap in the theory pointed
out in the Introduction.

\begin{exmp}
\label{exmp:obstacle}
Assume $k$ is perfect and $H$ is a $k$-defined subgroup of $G$.
Suppose $H$ is not $G$-completely reducible. Then $H$ is
not contained in any R-Levi subgroup of the optimal
destabilizing R-parabolic subgroup $P(H)$ of $G$.
Now $P(H)=P_\lambda$ for some $\lambda\in Y_k(G)$,
by Theorem~\ref{thm:potrat}(ii), so $H$ is not
$G$-completely reducible over $k$. This proves the
forward direction of \cite[Thm.~5.8]{BMR}.
The proof of the reverse direction given in
\emph{loc.\ cit.} is essentially just a special case
of the proof of Theorem~\ref{thm:sepRuconj}.

One deduces from the above as in \cite[Thm.~5.8]{BMR}
that if $k_1/k$ is a separable algebraic extension of
fields and $G$ and $H$ are $k$-defined, then,
under the hypothesis that $k$ is perfect,
$H$ is $G$-completely reducible over $k_1$
if and only if $H$ is $G$-completely reducible over $k$.
We answered the forward direction of Serre's question Theorem \ref{thm:serresquestion}
without the hypothesis that $k$ is perfect.  We cannot answer the
reverse direction by passing to $\ovl{k}$ using the argument in
the previous paragraph: for $H$ can be $G$-completely reducible
over $k$ (or $k_1$) and yet not $G$-completely reducible over
$\ovl{k}$, or vice versa (see Remark~\ref{rem:obstacle}).
To give a direct proof that the reverse implication holds, one would like to
associate an ``optimal destabilizing R-parabolic subgroup''
$P$ to $H$ having the property that $P$ is defined over $k_1$
and no R-Levi $k_1$-subgroup of $P$ contains $H$; optimality
should imply that $P$ is $\Gal(k_1/k)$-stable and hence $k$-defined, which would show that $H$ is not $G$-completely reducible over $k$.
We cannot take $P$ to be $P(H,c_\mu(H),k_1)$ for any $\mu\in Y_{k_1}(H)$,
because if $H$ is $G$-completely reducible,
then $P(H,c_\mu(H),k_1)$ is just $G$.
\end{exmp}

\begin{rem}
\label{rem:independence}
The construction of the optimal class of $k$-defined
cocharacters and the optimal destabilizing R-parabolic subgroup $P(H)$
from Definition \ref{defn:optpar}
depends on the choice of $n$ and the choice of norm $\left\|\,\right\|$.
In view of \cite[Sec.~7]{He} it is plausible that
this construction is independent of these choices.
Since the results we obtain here are
sufficient for our applications in the present
and subsequent sections, we do not pursue this question here and
leave it instead to a future study.
\end{rem}

\subsection{Counterparts for Lie subalgebras}
\label{sub:Lie}
There are counterparts to our results
for Lie subalgebras $\hh$ of the Lie algebra $\gg = \Lie G$ of $G$.
All of our results carry over with obvious modifications.
For instance, if $\hh$ is not $G$-completely reducible,
then there is an optimal destabilizing parabolic subgroup
$P$ of $G$ such that $\hh\subseteq \pp$ but $\hh\not\subseteq \frakl$
for any R-Levi subgroup $L$ of $P$, see Theorem \ref{thm:optparLie} below.
Many of the proofs are actually easier
in the Lie algebra case: for example, it often suffices to work in
connected $G$.
We just state the counterparts of
Theorems \ref{thm:orbclosconjcrit} and \ref{thm:optpar}
in this Lie algebra setting.
We leave the details of the proofs to the reader.

For a subgroup $H$ of $G$ we denote its Lie algebra $\Lie H$ by $\hh$.
We start with the analogue of
Definition \ref{def:gcr} in this setting, cf.\ \cite{mcninch};
see also \cite[Sec.\ 3.3]{BMRT:relative}.

\begin{defn}
\label{def:relLieGcr}
A subalgebra $\hh$ of $\gg$
is \emph{$G$-completely reducible} if for any
R-parabolic subgroup $P$ of $G$ such that $\hh\subseteq \pp$,
there is an R-Levi subgroup $L$ of $P$ such that $\hh\subseteq \frakl$.
\end{defn}

We require some standard facts concerning
Lie algebras of R-parabolic and R-Levi subgroups of $G$
(cf.~\cite[Sec.~2.1]{rich}).

\begin{lem}
\label{lem:liealgebrasofRpars}
For $\lambda\in Y(G)$, put $\pp_\lambda=\Lie(P_\lambda)$ and
$\frakl_\lambda=\Lie(L_\lambda)$.
Let $x\in\gg$. Then
\begin{enumerate}[{\rm(i)}]
\item $x\in\pp_\lambda$ if and only if $\underset{a\to 0}{\lim}\,
\lambda(a)\cdot x$ exists;
\item $x\in\frakl_\lambda$ if and only if $\underset{a\to 0}{\lim}\,
\lambda(a)\cdot x$ exists and equals $x$ if and only if $\lambda(\ovl{k})$ centralizes $x$;
\item $x\in\Lie(R_u(P_\lambda))$ if and only if $\underset{a\to 0}{\lim}\,
\lambda(a)\cdot x$ exists and equals $0$.
\end{enumerate}
\end{lem}

The map $c_\lambda : \pp_\lambda \to \frakl_\lambda$ given
by $x \mapsto \underset{a\to 0}{\lim}\, \lambda(a)\cdot x$ coincides with
the usual projection of $\pp_\lambda$ onto $\frakl_\lambda$.
In analogy with the construction for subgroups of $G$, we consider the action
of $G$ on $\gg^n$ by simultaneous adjoint action.

\begin{rem}
\label{rem:uniqueLiegcr}
 A statement analogous to Proposition \ref{prop:uniquegcr}
holds for Lie algebras: that is, given any Lie subalgebra $\hh$ of $\gg$,
we can find a uniquely defined $G$-conjugacy class of subalgebras of $\gg$
which contains $c_\lambda(\hh)$ for some  $\lambda \in Y(G)$, each member of which
is $G$-cr.
\end{rem}

\begin{thm}\
\label{thm:orbclosconjcritLie}
\begin{enumerate}[{\rm (i)}]
\item Let $n\in {\mathbb N}$, let $\tuple{h}\in \gg^n$ and
let $\lambda\in Y(G)$ such that
$\tuple{m}:=\lim_{a\ra 0}\lambda(a)\cdot \tuple{h}$ exists.
Then the following are equivalent:
\begin{enumerate}[{\rm (a)}]
\item $\tuple{m}$ is $G$-conjugate
to $\tuple{h}$;
\item $\tuple{m}$ is $R_u(P_\lambda)$-conjugate to $\tuple{h}$;
\item $\dim G\cdot\tuple{m}=\dim G\cdot\tuple{h}$.
\end{enumerate}
\item Let $\hh$ be a subalgebra of $\gg$ and let $\lambda\in Y(G)$.
Suppose $\hh\subseteq \pp_\lambda$ and set $\mm = c_\lambda(\hh)$.
Then $\dim C_G(\mm)\geq \dim C_G(\hh)$ and the following are equivalent:
\begin{enumerate}[{\rm (a)}]
\item $\mm$ is $G$-conjugate to $\hh$;
\item  $\mm$ is $R_u(P_\lambda)$-conjugate to $\hh$;
\item  $\hh$ is contained in the Lie algebra of
an R-Levi subgroup of $P_\lambda$;
\item $\dim C_G(\mm)=\dim C_G(\hh)$.
\end{enumerate}
\item Let $\hh$, $\lambda$ and $\mm$ be as in (ii) and let
$\tuple{h}\in \hh^n$ be a generating tuple of $\hh$. Then the
assertions in (i) are equivalent to those in (ii).
In particular, $\hh$ is $G$-completely reducible
if and only if $G\cdot\tuple{h}$ is closed in $\gg^n$.
\end{enumerate}
\end{thm}

Note that the final statement of Theorem \ref{thm:orbclosconjcritLie}(iii)
is \cite[Thm.\ 1(1)]{mcninch}.

If $\hh$ is a Lie subalgebra of $\gg$ and $\hh \subseteq \pp_\lambda$
for $\lambda \in Y(G)$,
then setting $\mm := c_\lambda(\hh)$ and
$S_n(\mm) := \overline{G \cdot \mm^n}$, we get an optimal class
$\Omega(\hh^n,S_n(\mm))$ of cocharacters, as in Definition \ref{def:SnM}.

\begin{thm}
\label{thm:optparLie}
Let $G$, $G'$ and $\left\|\,\right\|$ be as in
Theorem \ref{thm:optpar}.
Let $\hh$ be any subalgebra of $\gg$ and let
$n\in {\mathbb N}$ such that $\hh^n$ contains a generating tuple of $\hh$.
Let $\mm$ be a subalgebra of $\gg$ and suppose that
$\mm = c_\lambda(\hh)$ for some $\lambda\in Y_k(G)$
with $\hh \subseteq \pp_\lambda$. Put
$\Omega(\hh,\mm,k) := \Omega(\hh^n,S_n(\mm),k)$.
Then the following hold:
\begin{enumerate}[{\rm (i)}]
\item $P_\mu = P_\nu$ for all $\mu, \nu \in \Omega(\hh,\mm,k)$.
Let $P(\hh,\mm,k)$ denote the unique R-parabolic subgroup of $G$ so defined.
Then $\hh \subseteq \Lie(P(\hh,\mm,k))$ and $R_u(P(\hh,\mm,k))(k)$
acts simply transitively on $\Omega(\hh,\mm,k)$.
\item For $g\in G'(k)$ we have
$\Omega(g \cdot \hh,g \cdot \mm,k)=  g\cdot\Omega(\hh,\mm,k)$
and $P(g \cdot \hh,g \cdot \mm,k)=  gP(\hh,\mm,k)g^{-1}$.
If $g\in G(k)$ normalizes $\hh$ and stabilizes $S_n(\mm)$, then $g\in P(\hh,\mm,k)$.
\item If $\mu\in\Omega(\hh,\mm,k)$, then $\dim C_G(c_\mu(\hh))\ge\dim C_G(\mm)$.
If $\mm$ is $G$-conjugate to $\hh$, then $\Omega(\hh,\mm,k)=\{0\}$ and $P(\hh,\mm,k)= G$.
If $\mm$ is not $G$-conjugate to $\hh$, then $\hh$ is not contained
in the Lie algebra of any R-Levi subgroup of $P(\hh,\mm,k)$.
\end{enumerate}
\end{thm}

\begin{defn}
\label{defn:optparLie}
We call $\Omega(\hh,\mm,k)$ the
\emph{optimal class for $\hh$ with respect to $\mm$ over $k$}
and we call $P(\hh,\mm,k)$ the
\emph{optimal destabilizing R-parabolic subgroup for $\hh$ with respect to $\mm$ over $k$}.
Assume the $G$-conjugacy class given by
Remark \ref{rem:uniqueLiegcr}
contains a subalgebra $\mm$ of the form $c_\lambda(\hh)$
for some $\lambda\in Y_k(G)$.
Then we set $\Omega(\hh,k):= \Omega(\hh,\mm,k)$ and $P(\hh,k):= P(\hh,\mm,k)$.
Under this assumption we have, by
Remark \ref{rem:uniqueLiegcr}
and Theorem~\ref{thm:optparLie}, that $N_{G(k)}(\hh)$ is contained in $P(\hh,k)$
and that for $\mu\in\Omega(\hh,k)$, $c_\mu(\hh)$ is $G$-completely reducible.
So, by Theorem~\ref{thm:orbclosconjcritLie}(ii), if $\hh$ is not $G$-completely reducible,
then $\hh$ is not contained in the Lie algebra of
any R-Levi subgroup of $P(\hh,k)$. Note that, trivially, $P(\hh,k)=G$ if $\hh$
is $G$-completely reducible.
We call $\Omega(\hh,k)$ the
\emph{optimal class for $\hh$ over $k$}
and we call $P(\hh,k)$ the
\emph{optimal destabilizing R-parabolic subgroup for $\hh$ over $k$}.

Note that the assumption of the previous paragraph
is satisfied if $k$ is algebraically closed.
In that case we usually suppress
the $k$ argument and write simply $\Omega(\hh)$ and $P(\hh)$ instead; we refer to these as the
\emph{optimal class for $\hh$}
and the
\emph{optimal destabilizing R-parabolic subgroup for $\hh$},
respectively.
\end{defn}

\begin{exmp}\label{ex:GcrvsLieGcr}
As a further illustration of the power of our construction, we use Theorem~\ref{thm:optparLie} to give a short alternative
proof of \cite[Thm.\ 1(2)]{mcninch}, which states that $\hh=\Lie H$ is
$G$-completely reducible if $H$ is $G$-completely reducible.

Let $H$ be a subgroup of $G$. Assume that $\hh$ is not $G$-cr.
Let $P(\hh)$ be the optimal destabilizing R-parabolic
subgroup for $\hh$. 
By Theorem~\ref{thm:optparLie}(ii), $N_G(\hh) \subseteq P(\hh)$.
Clearly, $H\subseteq N_G(\hh)$.
Moreover, if $\mu\in \Omega(\hh)$ and $H\subseteq L_\mu$, then $\hh\subseteq \frakl_\mu$.  This is impossible by Theorem \ref{thm:optparLie}(iii), so $H$ is not $G$-cr.
Thus we can conclude that if $H$ is $G$-cr, then so is $\hh$.
\end{exmp}

\subsection{A special case of the Centre Conjecture}

In this final section we describe an application of optimal
destabilizing parabolic subgroups to the theory of spherical buildings
\cite{tits1}.
Suppose from now on that $G$ is connected.
Let $X=X(G,k)$ be the spherical Tits building  of
$G$ over $k$; then $X$ is a simplicial complex whose simplices
correspond to the $k$-defined parabolic subgroups of $G$.
The conjugation action of $G(k)$ on itself naturally induces an action of $G(k)$
on $X$.
We identify $X$ with its geometric realization.
A subcomplex $Y$ of $X$ is \emph{convex} if
whenever two points of $Y$ are not opposite in $X$,
then $Y$ contains the unique geodesic joining these points,
and $Y$ is \emph{contractible} if it has the homotopy type of a point.
The following is a version due to Serre of the so-called
Centre Conjecture of J.~Tits \cite[Sec.~2.4]{serre2}.
This has been proved by B.\ M\"uhlherr and J.\ Tits for spherical
buildings of classical type \cite{muhlherrtits}.

\begin{conj}
\label{conj:centre}
Let $Y$ be a convex and contractible subcomplex of $X$.
Then there is a point $y\in Y$ such that $y$ is fixed by
any automorphism of $X$ that stabilizes $Y$.
\end{conj}

A point $y \in Y$ whose existence is asserted in Conjecture \ref{conj:centre}
is frequently referred to as a ``natural centre'' or just ``centre'' of $Y$.
Our idea is to take as a centre of $Y$ the barycentre
of the simplex corresponding to the optimal destabilizing parabolic
subgroup in an appropriate sense.
This approach is not new; indeed, it was part of the motivation for Kempf's
paper \cite{kempf} on optimality (cf.\ \cite[p.\ 64]{mumford}).
We show how to make
this work to prove the
Centre Conjecture in the case that $Y$ is the
\emph{fixed point subcomplex $X^H$} for some subgroup $H$ of $G$,
where $X^H$ consists of all the simplices in $X$ corresponding to
parabolic subgroups containing $H$.
Note that $X^H$ is always a convex subcomplex of $X$ \cite[\S2.3.1]{serre1.5}.

\begin{thm}
\label{thm:X^N}
Suppose $G$ is semisimple and adjoint and $k$ is a perfect field.
Let $H$ be a subgroup of $G$ and suppose
that $Y:=X^H$ is contractible.  Then there is a point $y\in Y$
which is fixed by any element of $(\Aut G)(k)$ that stabilizes $Y$.
\end{thm}

\begin{proof}
Since we are assuming that $G$ is semisimple and defined over $k$, $\Aut G$
is an algebraic group also defined over $k$ \cite[5.7.2]{tits1}.
Since $G$ is adjoint, we can also view $G$ as a subgroup of $\Aut G$.
Let $K$ be
the intersection of all the $k$-defined parabolic subgroups of $G$
that contain $H$.
Then $K$ is $k$-defined, because $k$ is perfect, and
$X^{K}= Y$, cf.\ the proof of \cite[Thm.\ 3.1]{BMR:tits}.
Since $Y$ is contractible, $K$ is not $G$-cr over $k$, by a result of
Serre \cite[Sec.~3]{serre2}, and hence $K$ is not $G$-cr, by \cite[Thm.\ 5.8]{BMR}.

Now let $M$ be a representative of the unique $G$-conjugacy class of $G$-cr subgroups attached to
$K$ given by Proposition \ref{prop:uniquegcr}.
Let $P=P(K)$ be the optimal destabilizing parabolic subgroup for $K$ (over $\overline{k}$),
Definition~\ref{defn:optpar}.
Then $P$ is a parabolic subgroup of $G$ containing $K$,
by Theorem \ref{thm:optpar}(i), and $P$ is defined over $k$,
by Theorem \ref{thm:potrat}(ii),
so $P$ corresponds to a simplex of $Y$.
Moreover, any element of $(\Aut G)(k)$ that stabilizes $Y$
also normalizes $K$, and hence stabilizes the $G$-conjugacy class of $M$,
by Proposition \ref{prop:uniquegcr}(ii).
So any such automorphism normalizes $P$, by
Theorem~\ref{thm:optpar}(ii), with $G' = \Aut G$.
We can therefore take $y$ to be
the barycentre of the simplex corresponding to $P$.
\end{proof}

\begin{rem}
The assumptions that $G$ is semisimple and adjoint in Theorem \ref{thm:X^N} allow us
to apply our optimality results, because they ensure that $\Aut G$ is an algebraic
group and $G$ is a subgroup of $\Aut G$.
In the context of buildings, however, these assumptions are no loss: given any connected reductive $G$,
let $\Ad$ denote the adjoint representation of $G$.
Then the building of $G$ is isomorphic to the building of the
adjoint group $\Ad(G)$, and a subgroup $H$ of $G$ is $G$-cr if and only if the image of $H$
in $\Ad(G)$ is $\Ad(G)$-cr \cite[Lem.~2.12]{BMR}.
Moreover, all automorphisms of $X(G)$ that come from $\Aut G$ survive this transition
from $G$ to $\Ad(G)$.
\end{rem}

To establish that the Centre Conjecture holds for
subcomplexes of the form $Y=X^H$, we need to find a
centre $y\in Y$ which is fixed by all building automorphisms of
$X$ that stabilize $X^H$, not just the building automorphisms
that arise from algebraic automorphisms of $G$.
For most $G$, however, $\Aut X$ is generated by $\Aut G$ together
with field automorphisms:
see \cite[Cor.\ 5.9]{tits1} for more details.
We finish by showing how to deal with field automorphisms in some cases.

Recall that $\Gamma$ denotes the group $\Gal(k_s/k)$.
Following \cite[5.7.1]{tits1}, any $\gamma \in \Gamma$ induces
an automorphism of the building $X = X(G,\overline{k})$,
which we also denote by $\gamma$.
Recall that $\Gamma$ also acts on the set of cocharacters $Y(G)$
and we can ensure that the norm is invariant under this action
(i.e., the norm is $k$-defined
in the sense of Section~\ref{sec:uniform}).

\begin{thm}\label{thm:fieldauts}
Suppose $G$ is connected.
Let $X = X(G,\overline{k})$ be the building of $G$ over the algebraic closure of $k$.
Let $H$ be a subgroup of $G$ and suppose
that $Y:=X^H$ is contractible.
Let $\Gamma_Y$ denote the subgroup of $\Gamma$ that stabilizes $Y$.
Then there is a point $y\in Y$ which is fixed by any element of $\Gamma_Y$.
\end{thm}

\begin{proof}
As in the proof of Theorem \ref{thm:X^N}, let $K$ be the intersection
of the parabolic subgroups corresponding to simplices in $Y$.
Then $Y=X^K$, and since $Y$ is stabilized by all $\gamma \in \Gamma_Y$,
we have $\gamma\cdot K = K$ for all $\gamma \in \Gamma_Y$.
Let $\lambda$ and $M = c_\lambda(K)$ be as in Proposition \ref{prop:uniquegcr}.  Choose $n\in {\mathbb N}$ such that $K$ admits a generic $n$-tuple ${\mathbf k}\in K^n$.
Then $S_n(M)$ is $\Gamma_Y$-stable by Proposition \ref{prop:uniquegcr}(iii).
Because the norm is $\Gamma$-invariant, for any $\lambda\in \Lambda(K^n)$ and any $\gamma\in \Gamma_Y$, we have
$$ \frac{\alpha_{S_n(M),K^n}(\gamma\cdot \lambda)}{\left\| \gamma\cdot \lambda \right\|}= \frac{\alpha_{S_n(M),K^n}(\lambda)}{\left\| \lambda \right\|}. $$
It follows that the optimal parabolic subgroup $P(K)$ for $K$ is stabilized by $\Gamma_Y$.  We can therefore take $y$ to be the barycentre of the simplex corresponding to $P(K)$.
\end{proof}

\begin{rems}
(i). Combining Theorem \ref{thm:X^N} and Theorem \ref{thm:fieldauts} goes a long way towards
proving the full version of Tits' Centre Conjecture for subcomplexes of the form $X^H$ in many cases.
For example, if $G$ is a split simple group of adjoint type defined over a finite field $k$,
then,
with a few exceptions, the automorphism group of $X(G,\overline{k})$ is a split extension of $\Aut G$ by
the automorphism group of the field $\overline{k}$ (see \cite[Cor.\ 5.10]{tits1}),
and the results above show how to deal with many of these automorphisms.

(ii). Theorems \ref{thm:X^N} and \ref{thm:fieldauts} improve on \cite[Thm.\ 3.1]{BMR:tits}.
\end{rems}


\bigskip
{\bf Acknowledgements}:
The authors acknowledge the financial support of EPSRC Grant EP/C542150/1,
Marsden Grant UOC0501 and
the DFG-priority programme SPP1388 ``Representation Theory''.
Part of the research for this paper was carried out while the
authors were staying at the Mathematical Research Institute
Oberwolfach supported by the ``Research in Pairs'' programme.
Also, part of this paper was written during a stay of the three
first authors at the Isaac Newton Institute for Mathematical
Sciences, Cambridge during the ``Algebraic Lie Theory'' Programme in 2009.
Finally, we are grateful to the referee for carefully reading the manuscript and for some suggestions.

\end{document}